\documentclass[11pt]{amsart}

\usepackage[latin1]{inputenc}
\usepackage{amsmath,amssymb,amsthm}
\usepackage{t1enc}

\usepackage[english]{}
\usepackage{graphicx}
\usepackage{epsfig}
\usepackage{graphicx}
\usepackage{multirow}
\usepackage{fancyhdr}
\usepackage{subfigure}
\usepackage{color}

\usepackage{pdfsync}
\usepackage{dsfont}
\renewcommand{\geq}{\geqslant}

\def\N{\textrm{I\kern-0.21emN}}

\def\R{\textrm{I\kern-0.21emR}}
\def\Z{\mathbb{Z}}
\def\P{\mathbb{P}}

\newcommand{\ds}{\displaystyle}

\newcommand{\dive}{\mathop{\rm div}\nolimits}
\newcommand{\ts}{T_{*}}

\newcommand{\vsd}{\vskip 0.2cm}
\def\virgp{\raise 2pt\hbox{,}}
\def\cdotpv{\raise 2pt\hbox{;}}
\def\eqdefa{\buildrel\hbox{{\rm \footnotesize def}}\over =}

\newtheorem{theorem}{Theorem}[section]

\newtheorem{prop}[theorem]{Proposition}

\newtheorem{lemma}[theorem]{Lemma}

\theoremstyle{remark}
\newtheorem{remark}{Remark}[section]

\theoremstyle{definition}
\newtheorem{definition}{Definition}[section]

\theoremstyle{definition}

\theoremstyle{definition}

\title{About the possibility of minimal blow up for Navier-Stokes solutions with data in $\dot{H}^s(\R^3)$}
\author{Eugénie Poulon}
\date\today

\begin{document}

\synctex=1
\newcommand{\nfont}{\fontshape{n}\selectfont}

\address{({\nfont{Eug\'enie Poulon}}) Laboratoire Jacques-Louis Lions - UMR 7598, Universit\'e Pierre et Marie Curie, Bo\^{i}te courrier 187, 4 place Jussieu, 75252 Paris Cedex 05, France}

\email{poulon@ann.jussieu.fr} 


\begin{abstract}
\noindent Considering initial data in $\dot{H}^s$, with $\frac{1}{2} < s < \frac{3}{2}$, this paper is devoted to the study of possible blowing-up Navier-Stokes solutions such that $\ds{(\ts(u_{0}) -t)^{\frac{1}{2} (s- \frac{1}{2})} \,\, \| u \|_{\dot{H}^s} }$ is bounded. Our result is in the spirit of the tremendous works of L. Escauriaza, G. Seregin, and V. $\breve{\mathrm{S}}$ver$\acute{\mathrm{a}}$k and I. Gallagher, G. Koch, F. Planchon, where they proved there is no blowing-up solution which remain bounded in $L^3(\R^3)$. The main idea is that if such blowing-up solutions exist, they satisfy critical properties. 
\end{abstract}
\maketitle

\section{Introduction and statement of main result}
\medbreak
\noindent    We consider the Navier-Stokes system for incompressible viscous fluids evolving in the whole space $\R^{3}$. Denoting by $u$ the velocity, a vector field in $\R^3$, by $p$ in $\R$ the pressure function, the Cauchy problem for the homogeneous incompressible Navier-Stokes system is given by 
\begin{equation} \left \lbrace \begin {array}{ccc}  \partial_{t}u + u\cdot\nabla{u}-\Delta{u}&=&-\nabla{p}\\ \dive u&=&0\\
 u_{|t=0}&=&u_{0}.\\ \end{array}
\right.
\end{equation}
\vsd

\noindent We recall a crucial property of the Navier-Stokes equation : the scaling invariance. Let us define the operator
\begin{equation}
\label{notation}
\begin{split}
\forall \alpha \in \R^{+},\,\, \forall \lambda \in \R^{+}_{*},\,\,  &\forall x_{0} \in \R^3,\,\,\,\, \Lambda^{\alpha}_{\lambda,x_{0}}\,u(t,x) \eqdefa \frac{1}{\lambda^{\alpha}}u\Bigl(\frac{t}{\lambda^2} \virgp \frac{x-x_{0}}{\lambda}  \Bigr).\\
&\hbox{If} \,\,\, \alpha =1,\,\,\, \hbox{we note} \,\,\,  \Lambda^{1}_{\lambda,x_{0}} =  \Lambda_{\lambda,x_{0}}. 
\end{split}
\end{equation}

\noindent Clearly, if $u$ is smooth solution of Navier-Stokes system on ~$[0,T] \times \R^3$ with pressure $p$ associated with the initial data $u_{0}$, then, for any positive $\lambda$, the vector field and the pressure 
$$ u_{\lambda}  \eqdefa \Lambda_{\lambda,x_{0}}\,u \quad \hbox{and} \quad p_{\lambda}  \eqdefa \Lambda^{2}_{\lambda,x_{0}}\,p $$
is a solution of Navier-Stokes system on the interval $[0,\lambda^2T] \times \R^3$, associated with the initial data $$u_{0,\lambda}  = ~\Lambda_{\lambda,x_{0}}\, u_{0}.$$ 
This leads to the definition of scaling invariant space.
\begin{definition}\sl{
A Banach space $X$ is said to be scaling invariant (or also critical),  if its norm is invariant under the scaling transformation defined by $u \mapsto u_{\lambda}$
$$ || u_{\lambda} ||_{X} = || u ||_{X}.$$}
\end{definition}
\noindent Let us give some exemples of critical spaces in dimension $3$
$$ \dot{H}^{\frac{1}{2}}(\R^3) \hookrightarrow L^3(\R^3) \hookrightarrow \dot{B}^{-1 +\frac{3}{p}}_{p,\infty}(\R^3)_{p < \infty} \hookrightarrow \mathcal{BMO}^{-1}(\R^3) \hookrightarrow \dot{B}^{-1 }_{\infty,\infty}(\R^3).$$

\noindent The framework of this work is functional spaces which are above the natural scaling of Navier-Stokes equations. More precisely, our statements will take place in some Sobolev and Besov spaces, with a regularity index $s$ such that $\ds{\frac{1}{2} < s < \frac{3}{2}\cdotp}$\\

\noindent \textbf{Notations.} We shall constantly be using the following simplified notations:  $$L^{\infty}_{T}(\dot{H}^s)  \eqdefa L^{\infty}([0,T],\dot{H}^s)\quad \hbox{and} \quad L^{2}_{T}(\dot{H}^{s+1})  \eqdefa L^{2}([0,T],\dot{H}^{s+1}),$$
\noindent and the relevant function space we shall be working with in the sequel is $$X^{s}_{T} \eqdefa L^{\infty}_{T}(\dot{H}^s)\, \cap\, L^{2}_{T}(\dot{H}^{s+1}),\quad \hbox{endowed with the norm} \quad\| u \|^2_{X^{s}_{T}}  \eqdefa \| u \|^2_{L^{\infty}_{T}(\dot{H}^s)} + \| u \|^2_{L^{2}_{T}(\dot{H}^{s+1})}.$$
\vskip 0.2cm
\noindent Let us start by recalling the local existence theorem for data in the Sobolev space $\dot{H}^s$. 
\begin{theorem} \sl{
\label{key theorem}
Let $u_{0}$ be in $\dot{H}^s$, with $\ds{\frac{1}{2} < s < \frac{3}{2}}\cdotp$ Then there exists a time $T$ and there exists a unique solution $NS(u_{0})$ such that  
$\ds{NS(u_{0}) \quad\hbox{belongs to} \quad L^{\infty}_{T}(\dot{H}^s) \cap L^{2}_{T}(\dot{H}^{s+1})}$.\\
Moreover,  denoting by $\ts(u_{0})$ the maximal time of existence of such a solution, there exists a positive constant ~$c$ such that  
\begin{equation}
\label{relation Ac et ts}
 \ts(u_{0})\,\, \| u_{0} \|^{\sigma_{s}}_{\dot{H}^s} \geqslant c,\quad \hbox{with} \quad \sigma_{s} \eqdefa \frac{1}{\frac{1}{2}(s-\frac{1}{2})}\cdotp
\end{equation}}
\end{theorem}
\medbreak
\begin{remark}
Throughout this paper, we will adopt the useful notation $NS(u_{0})$ to mean the maximal solution of the Navier-Stokes system, associated with the initial data $u_{0}$. Notice that our whole work relies on the hypothesis there exists some blowing up $NS$-solutions, e.g some $NS$-solutions with a finite lifespan~$\ts(u_{0})$. This is still an open question.
\end{remark}

\begin{remark}
We point out that the infimum of the quantity $\ds{\ts(u_{0})\,\, \| u_{0} \|^{\sigma_{s}}_{\dot{H}^s}}$ exists and is positive (because of the constant $c$).  It has been proved in \cite{P} that there exists some intial data which reach this infimum and that the set of such data is compact, up to dilations and translations. 
\end{remark}

\begin{remark}
\noindent Theorem \ref{key theorem} implies there exists a constant $c>0$, such that
\begin{equation}
\label{remark c}
(\ts(u_{0}) - t)\, \| NS(u_{0})(t) \|^{\sigma_{s}}_{\dot{H}^s} \geqslant c,
\end{equation}
\noindent and thus we get in particular the blow up of the $\dot{H}^s$-norm 
$$ \lim_{t \to \ts(u_{0})} \| NS(u_{0})(t)\|^{\sigma_{s}}_{\dot{H}^s} = +\infty.$$
\end{remark}
\vskip 0.2cm
\noindent Our motivation here is to wonder if there exist some Navier-Stokes solutions which stop living in finite time (e.g~$\ts(u_{0}) < \infty$) and which blows up at a minimal rate, namely: there exists a positive constant~$M$ such that~$\ds{(\ts(u_{0}) - t)\, \| NS(u_{0}) \|^{\sigma_{s}}_{\dot{H}^s} \leqslant M}$. In others terms,  
\\
\begin{center}
\noindent \textit{Question:} \sl{ Does there exist some blowing up $NS$-solutions  such that~$\ds{(\ts(u_{0}) - t)\, \| NS(u_{0}) \|^{\sigma_{s}}_{\dot{H}^s} \leqslant M}$ ?\\
If yes, what do they look like ? 
}
\end{center}
\vskip 0.3cm
\noindent We assume an affirmative answer and we search to characterize such solutions.\\
\begin{center}
\noindent \textit{Hypothesis $\mathcal{H}$:} \sl{ There exist some blowing up $NS$-solutions such that $\ds{(\ts(u_{0}) - t)\, \| NS(u_{0}) \|^{\sigma_{s}}_{\dot{H}^s} \leqslant M}$.
}
\end{center}
\vskip 0.1cm
\noindent Notice that a very close question to this one is to prove that
$$ \hbox{If} \quad \ts(u_{0}) < \infty, \quad \hbox{does} \quad \limsup_{t \to \ts(u_{0}) } (\ts(u_{0}) - t)\, \| NS(u_{0})(t) \|^{\sigma_{s}}_{\dot{H}^s} = +\infty \quad ?$$

\noindent We underline that this question about blowing-up Navier-Stokes solutions has been highly developed in the context of critical spaces, namely $\dot{H}^{\frac{1}{2}}(\R^3)$ and~$L^{3}(\R^3)$. Indeed, L. Escauriaza, G. Seregin and V. $\breve{\mathrm{S}}$ver$\acute{\mathrm{a}}$k showed in the fundamental work \cite{ESS} that any "Leray-Hopf" weak solution which remains bounded in~$L^{3}(\R^3)$ can not develop a singularity in finite time. Alternatively, it means that 
\begin{equation}
\hbox{If} \,\, \ts(u_{0}) <  +\infty,\,\,\,\, \hbox{then} \,\, \limsup_{t \to \ts(u_{0})} \, \| NS(u_{0})(t) \|_{L^3} = +\infty.
\end{equation}
\noindent I. Gallagher, G. Koch and F. Planchon revisited the above criteria in the context of mild Navier-Stokes solutions. They proved in \cite{GKP} that strong solutions which remain bounded in $L^{3}(\R^3)$, do not become singular in finite time. To perform it, they develop an alternative viewpoint : the method of "critical elements" (or "concentration-compactness"), which was introduced by C. Kenig and F. Merle to treat critical dispersive equations. Recently, same authors extend the method in~\cite{GKP2} to prove the same result in the case of the critical Besov space~$\ds{\dot{B}^{-1 +\frac{3}{p}}_{p,q}(\R^3)}$, with~$\ds{3<p,q<\infty}$. Notice the work of J.-Y.Chemin and F. Planchon in~\cite{CP}, who gives the same answer in the case of the Besov space~$\ds{\dot{B}^{-1 +\frac{3}{p}}_{p,q}(\R^3)}$, with $\ds{3<p<\infty}$,\, $q<3$ and with an additional regularity assumption on the data. To conclude the non-exhaustive list of blow up results, we mention the work of C. Kenig and G. Koch who carried out in~\cite{KK} such a program of critical elements for solutions in the simpler case $\dot{H}^{\frac{1}{2}}(\R^3)$. More precisely, they proved  for any data~$u_{0}$ belonging to the smaller critical space  ~$\dot{H}^{\frac{1}{2}}(\R^3)$, 
\begin{equation}
\hbox{If} \,\, \ts(u_{0}) <  +\infty,\,\,\,\, \hbox{then} \,\, \lim_{t \to \ts(u_{0})} \, \| NS(u_{0})(t) \|_{\dot{H}^{\frac{1}{2}}} = +\infty.
\end{equation}

\noindent In our case (remind : we consider Sobolev spaces $\dot{H}^{s}(\R^3)$ with $\ds{ \frac{1}{2} < s < \frac{3}{2}}$ which are non-invariant under the natural scaling of Navier-Stokes equations), we can not expect to prove our result in the same way, because of the scaling. Indeed, a similar proof leads us to define the critical quantity $M^{\sigma_{s}}_{c}$ 
$$ M^{\sigma_{s}}_{c} = \sup \bigl\{ A>0,\,\, \sup_{t < \ts(u_{0}) } (\ts(u_{0}) - t)\, \| NS(u_{0}) \|^{\sigma_{s}}_{\dot{H}^s} \leqslant A\,\,\,\, \Rightarrow\,  \ts(u_{0}) = +\infty   \bigr\}.$$
But unfortunately, such a point of view makes no sense, owing to the meaning of~$(\ts(u_{0}) - t)$ when~$\ts(u_{0}) ~= +\infty$. We have to proceed in an other way and it may be removed by defining a new object~$M^{\sigma_{s}}_{c}$ 
$$ M^{\sigma_{s}}_{c}  \eqdefa  \inf_{\substack{u_{0} \in \dot{H}^s \\ \ts(u_{0}) < \infty    }}   \bigl\{  \limsup_{t \to \ts(u_{0})}(\ts(u_{0}) - t)\, \| NS(u_{0})(t) \|^{\sigma_{s}}_{\dot{H}^s} \bigr\}.$$
\noindent Clearly, (\ref{remark c}) implies that $M^{\sigma_{s}}_{c} $ exists and is positive. As we have decided to work under hypothesis $\mathcal{H}$, \textit{a fortiori}, this implies that $M^{\sigma_{s}}_{c} $ is finite. The definition below is the key notion of critical solution in this context. 
\begin{definition}\sl{(Sup-critical solution)\\
Let $u_{0}$ be an element in $\dot{H}^s$. We say that $u = NS(u_{0})$ is a sup-critical solution if $NS(u_{0})$ satisfies the two following assumptions:
$$\ts(u_{0}) < \infty \quad \hbox{and} \quad  \limsup_{t \to \ts(u_{0})} (\ts(u_{0}) - t)\,  \| NS(u_{0})(t) \|^{\sigma_{s}}_{\dot{H}^s} =\, M^{\sigma_{s}}_{c}.$$ 
}\end{definition}
\noindent A natural question is to know if such elements exist. The statement given below gives an affirmative answer and provides a general procedure to build some sup-critical solutions. Our main result follows. 

\begin{theorem}\sl{(Key Theorem)\\
\label{Big key theorem}
Let us assume that there exists $u_{0}$ in $\dot{H}^s $ and  $M$ in $\R^{+}_{*}$ such that
$$\ts(u_{0}) < \infty \quad \hbox{and} \quad   (\ts(u_{0}) - t)\,  \| NS(u_{0})(t) \|^{\sigma_{s}}_{\dot{H}^s} \leqslant M.$$ 
\noindent Then, there exists $\Phi_{0} \in \dot{H}^s \cap \dot{B}^{\frac{1}{2}}_{2,\infty}$ such that $\Phi \eqdefa NS(\Phi_{0})$ is a sup-critical solution, blowing up at time $1$,  such that \begin{equation}
\label{key theorem point 1}
\sup_{ \tau < 1} \,(1 - \tau)\,  \| NS(\Phi_{0})(\tau) \|^{\sigma_{s}}_{\dot{H}^s} =\,  \limsup_{\tau\to 1} (1 - \tau)\,  \| NS(\Phi_{0})(\tau) \|^{\sigma_{s}}_{\dot{H}^s} =\, M^{\sigma_{s}}_{c}.
\end{equation}
%
\noindent In addition, there exists a positive constant $C$ such that
\begin{equation}
\label{key theorem point 2}
\hbox{and for any} \quad \tau <1, \quad \| NS(\Phi_{0})(\tau) \|_{\dot{B}^{\frac{1}{2}}_{2,\infty}} \leqslant C, 
\end{equation}
\noindent where the Besov norm (for regularity index $0<\alpha <1$) is defined by
$$ \| u \|_{\dot{B}^{\alpha}_{2,\infty}} \eqdefa \sup_{x \in \R^d} \,\, \frac{\| u(\cdotp -x) - u\|_{L^{2}}}{|x| ^{\alpha}}\cdotp$$
}\end{theorem}
\noindent We postpone the proof of (\ref{key theorem point 1}) of the Key Theorem \ref{Big key theorem} to the next section. The proof of (\ref{key theorem point 2}) will be given in Section $5$. We stress on the fact that (\ref{key theorem point 2}) is somewhat close to a question raised by the paper of I. Gallagher, G. Koch and F. Planchon \cite{GKP2}, in which they prove that for any initial data in the critical Besov space $\dot{B}^{-1+ \frac{3}{p}}_{p, q}$, with $3<p,q<\infty$, the $NS$-solution, (the lifespan of which is assumed finite) becomes unbounded at the blow-up time. Let us say a few words about the limit case~$\dot{B}^{-1+ \frac{3}{p}}_{p, \infty}$. We may wonder if the result holds in the limit case $q=\infty$. As far as the author is aware, the answer is still open. Actually, if it holds, \textit{a fortiori} it holds in the smaller space $\dot{B}^{\frac{1}{2}}_{2, \infty}$, by vertue of the embedding $\ds{\dot{B}^{\frac{1}{2}}_{2, \infty}  \hookrightarrow \dot{B}^{-1+ \frac{3}{p}}_{p, \infty}}$. In others terms, it would mean there is no blowing-up solution, bounded in the critical space $\dot{B}^{\frac{1}{2}}_{2, \infty}$. This is related to the concern of our paper since we build some blowing-up solutions bounded in this critical space, under the assumption of blow up at minimal rate. We mention the very interesting work  of H. Jia and V. $\breve{\mathrm{S}}$ver$\acute{\mathrm{a}}$k \cite{JV}, where they prove that $-1$-homogeneous initial data generate global $-1$-homogeneous  solutions. Unfortunately, the uniqueness of such solutions is not guaranteed.

\medbreak
\section{Existence of sup-critical solutions}
\vskip 0.5cm
\noindent The goal of this section is to give a partial proof of Key Theorem \ref{Big key theorem}. It relies on the two Lemmas below. 

\begin{lemma}\sl{(Existence of sup-critical solutions in $\dot{H}^s $)\\
\label{general lemma for critical element }
Let $(v_{0,n})_{n \in \N}$ be a bounded sequence in $\dot{H}^s$ such that
\begin{equation}
\tau^{*}(v_{0,n}) = 1 \quad \hbox{ and} \quad \hbox{ for any} \,\,\, \tau < 1, \quad (1-\tau)\, \| NS(v_{0,n})(\tau,\cdotp) \|^{\sigma_{s}}_{\dot{H}^s} \leqslant\,\, M^{\sigma_{s}}_{c} + \varepsilon_{n},
\end{equation}
where $\varepsilon_{n}$ is a generic sequence which tends to $0$ when $n$ goes to $+\infty$.\\
Then, there exists $\Psi_{0}$ in $\dot{H}^s$ such that $\Psi \eqdefa NS(\Psi_{0})$ is a sup-critical solution blowing up at time $1$ and satisfies
\begin{equation}
\sup_{\tau < 1} (1 - \tau)\,  \| NS(\Psi_{0})(\tau) \|^{\sigma_{s}}_{\dot{H}^s} \,=\,    \limsup_{\tau \to 1} (1 - \tau)\,  \| NS(\Psi_{0})(\tau) \|^{\sigma_{s}}_{\dot{H}^s} =\, M^{\sigma_{s}}_{c}.
\end{equation}
\noindent Moreover, the initial data of such element is a weak limit of the sequence $(v_{0,n})$ translated, e.g 
\begin{equation}
\ds{\exists\,\, (x_{0,n})_{n \geqslant 0}, \quad   v_{0,n}(\cdotp + x_{0,n} ) \rightharpoonup_{n \to +\infty} \Psi_{0}}.
\end{equation} 
}
\end{lemma}
\noindent The proof of Lemma \ref{general lemma for critical element } will be the purpose of Section $4$. It relies essentially on scaling argument and profile theory, which will be introduced in the next Section $3$.
\medbreak
 
\begin{lemma}\sl{(Fluctuation estimates)\\
\label{fluctuation lemma}
\noindent Let $u=NS(u_{0})$ be a NS-solution associated with a data $ u_{0} \in  \dot{H}^s $, with $\ds{\frac{1}{2} < s < \frac{3}{2}}$, such that 
$$ (\ts(u_{0}) - t)^{\frac{1}{\sigma_{s}}}\,  \| NS(u_{0})(t) \|_{\dot{H}^s} \leqslant M.$$
\noindent Then, the following estimates on the fluctuation part $\ds{B(u,u)(t) \eqdefa u - e^{t\, \Delta}u_{0} }$ yield
\begin{equation}
\hbox{ for any}\quad s <s' < 2s-\frac{1}{2}, \quad (\ts(u_{0})- t)^{\frac{1}{\sigma_{s'}}}\,  \| B(u,u)(t) \|_{\dot{H}^{s'}}\, \leqslant F_{s'}(M^2)
\end{equation}
\noindent Moreover, for the critical case $\ds{ = \frac{1}{2}}$, we have
\begin{equation}
\| B(u,u)(t) \|_{\dot{B}^{\frac{1}{2}}_{2,\infty}}\, \leqslant C\,M^2.
\end{equation}
}
\end{lemma} 
\noindent The proof of this lemma is postpone to Section $8$. It merely stems from product laws in Besov spaces, interpolation inequalities and from judicious splitting into low and high frequencies  in the following sense $$(\ts-t)2^{2j} \leqslant 1 \quad  \hbox{and} \quad  (\ts-t)2^{2j} \geqslant 1.$$ 
\begin{remark}
Let us point out that estimates  of Lemma \ref{fluctuation lemma} do not hold if $\ds{0<\alpha < \frac{1}{2}}$, owing to low frequencies. Indeed, arguments similar to the ones used in the proof of Lemma \ref{fluctuation lemma} lead only to the following estimate  $$\| B(u,u)(t) \|_{\dot{B}^{\alpha}_{2,\infty}}\, \leqslant C\, M^2\, \ts(u_{0})^{\frac{1}{2}(\alpha-\frac{1}{2})}.$$
\end{remark}

\vskip 0.3cm
 \noindent \textit{Partial proof of Key Theorem \ref{Big key theorem} }\\
 \noindent In all this text, we denote by $(\varepsilon_{n}) $ a non increasing sequence, which tends to $0$, when $n$ tend to $+\infty$. \\
\noindent $\bullet$ Step $1$ : Existence of sup-critical elements in $ \dot{H}^s$, with $\ds{\frac{1}{2} < s < \frac{3}{2}}\cdotp$\\
\noindent Let us consider the sequence $\ds{(M_{c} + \varepsilon_{n})_{n \geqslant 0}}$. By definition of $M_{c}$, there exists a sequence $(u_{0,n})$ belonging to ~$\dot{H}^s$, with a finite lifespan $\ts(u_{0,n})$,  such that for any ~$t < \ts(u_{0,n})$  : 
$$ \limsup_{t \to \ts(u_{0})} (\ts(u_{0,n}) - t)\, \| NS(u_{0,n}) \|^{\sigma_{s}}_{\dot{H}^s} \leqslant M^{\sigma_{s}}_{c} + \varepsilon_{n}.$$
By definition of $\ds{\limsup}$, there exists a nondecreasing sequence of time ~$t_{n}$, converging to $\ts(u_{0})$, such that 
\begin{equation}
\label{limsup tn}
\forall t \geqslant t_{n},\,\, (\ts(u_{0,n}) - t)\, \| NS(u_{0,n})(t,x) \|^{\sigma_{s}}_{\dot{H}^s} \leqslant M^{\sigma_{s}}_{c} + \varepsilon_{n}.
\end{equation} 
By rescaling, we consider the sequence $$v_{0,n}(y) = \bigl(\ts(u_{0,n}) - t_{n}\bigr)^\frac{1}{2} \, NS(u_{0,n})\bigl( t_{n},(\ts(u_{0,n}) - t_{n}\bigr)^\frac{1}{2}\,y \bigr).$$
\noindent and we have
\begin{equation}
\label{rescaling donnee initiale}
\begin{split}
 \|v_{0,n} \|^{\sigma_{s}}_{\dot{H}^s} &= \bigl(\ts(u_{0,n})-t_{n}\bigr)\, \|NS(u_{0,n})(t_{n}) \|^{\sigma_{s}}_{\dot{H}^s}. \\
\end{split}
\end{equation}
\noindent By vertue of (\ref{limsup tn}), the sequence $(v_{0,n})_{n \geqslant 1}$ is bounded $\bigl(\hbox{by}\,\, \ds{M^{\sigma_{s}}_{c} + \varepsilon_{0}}\bigr)$ in the space $\dot{H}^s$. Moreover, such a sequence generates a Navier-Stokes solution, which keeps on living until the time~$\tau^* = 1$ and satisfies 
\begin{equation}
\label{rescaling solution}
\begin{split}
 NS(v_{0,n})(\tau,y)&= \bigl(\ts(u_{0,n})-t_{n}\bigr)\,^\frac{1}{2} \, NS(u_{0,n})\bigl( t_{n} + \tau\,\bigl(\ts(u_{0,n})-t_{n}\bigr)\, ,\bigl(\ts(u_{0,n})-t_{n}\bigr)\,^\frac{1}{2}\,y \bigr).
\end{split}
\end{equation}
\noindent We introduce $\ds{\widetilde{t_{n}} = t_{n} + \tau\,\bigl(\ts(u_{0,n})-t_{n}\bigr)\,}$. Notice that, because of scaling, an easy computation yields
\begin{equation}
\label{proposition comparative}
(1-\tau)\, \|NS(v_{0,n})(\tau) \|^{\sigma_{s}}_{\dot{H}^s} = \bigl(\ts(u_{0,n}) - \widetilde{t_{n}}\bigr) \, \|NS(u_{0,n})\bigl( \widetilde{t_{n}} \bigr)\|^{\sigma_{s}}_{\dot{H}^s}.
\end{equation}
\noindent As $\widetilde{t_{n}} \geqslant t_{n}$ for any $n$ (by definition of $\widetilde{t_{n}}$) we combine (\ref{proposition comparative}) with (\ref{limsup tn}) and we get, for any $\tau \in [0,1[$,
$$(1-\tau)\| NS(v_{0,n})(\tau,x) \|^{\sigma_{s}}_{\dot{H}^s} \leqslant M^{\sigma_{s}}_{c} + \varepsilon_{n}.$$
\noindent The sequence $(v_{0,n})$ satisfies the hypothesis of Lemma \ref{general lemma for critical element }. Applying it, we build a sup-critical solution $ \Phi = NS(\Psi_{0}) $ in $\dot{H}^s$ which blows up at time $1$, e.g
$$  \limsup_{\tau \to 1} (1 - \tau)\,  \| NS(\Psi_{0})(\tau) \|^{\sigma_{s}}_{\dot{H}^s} =\, M^{\sigma_{s}}_{c}.$$
\noindent This proves the first part of the statement of Theorem \ref{Big key theorem}. \\

\noindent $\bullet$ Step $2$ : Existence of sup-critical elements in $\dot{H}^s \cap \dot{B}^{\frac{1}{2}}_{2,\infty} \cap \dot{H}^{s'} $, with $s$ and $s'$ such that $\ds{s<s'< 2s -\frac{1}{2}\cdotp}$\\ This will be proved in Section $6$. Notice that proving that $NS(\Psi_{0})$ is bounded in the Besov space~$\dot{B}^{\frac{1}{2}}_{2,\infty}$ is equivalent to prove that~$\Psi_{0}$ belongs to $\dot{B}^{\frac{1}{2}}_{2,\infty}$, since, by vertue of Lemma \ref{fluctuation lemma}, the fluctuation part is bounded in~$\dot{B}^{\frac{1}{2}}_{2,\infty}$ and obviously we have
$$ \| NS(\Psi_{0})(t)\|_{\dot{B}^{\frac{1}{2}}_{2,\infty}}  \leqslant\| NS(\Psi_{0})(t) \,\,-\,\, e^{t\Delta}\Psi_{0}\|_{\dot{B}^{\frac{1}{2}}_{2,\infty}} \quad + \quad  \| e^{t\Delta}\Psi_{0}\|_{\dot{B}^{\frac{1}{2}}_{2,\infty}}.$$

\medskip
\noindent The paper is structured as follows. In Section $3$, we recall the main tools of this paper. Essentially, it deals with the profile theory of P. G\'erard \cite{PG} and a structure lemma concerning a $NS$-solution associated with a sequence which satisfies hypothesis of profile theory. We also recall some basics facts on Besov spaces. \\ 
In Section $4$, we are going to establish the proof of crucial Lemma \ref{general lemma for critical element }, which provides the proof of the first part of Theorem \ref{Big key theorem} : there exists some sup-critical elements in $\dot{H}^s$. The second part of the proof is postponed in Section $6$, where we build some sup-critical elements not only in $\dot{H}^s$, but also in others spaces, such as $\dot{B}^{\frac{1}{2}}_{2,\infty} $ and $\dot{B}^{s'}_{2,\infty} $, with $\ds{s<s'< 2s -\frac{1}{2}\cdotp}$ To carry out this, we need some estimates on the fluctuation part of the solution, which will be provided in Section $5$.\\ Then in Section $7$, we give an analogue sup-inf critical criteria. It turns out that among sup-critical solutions, there exists some of them which are sup-inf-critical in the sense of they reach the biggest infimum limit. Section $8$ is devoted to the proof of Lemma \ref{lemme allure de la solution}, which gives the structure of a Navier-Stokes solution associated with a bounded sequence of data in $\dot{H}^s$. We recall to the reader that such structure result has been partially proved in \cite{P}, except for the orthogonality property of Navier-Stokes solution in $\dot{H}^s$-norm. As a result, we give the proof of such a property, after reminding the ideas of the complete proof. 
  
\medbreak

\section{Profile theory and Tool Box}
\noindent We recall the fundamental result due to P. G\'erard : the profile decomposition of a bounded sequence in the Sobolev space~$\dot{H}^s$. The original motivation of this theory was the desciption, up to extractions, of the defect of compactness in Sobolev embeddings (see for instance the pionneering works of P.-L. Lions in \cite{PLL}, \cite{PLL2} and H. Brezis, J.-M. Coron in \cite{BC}. Here, we will use the theorem of P. G\'erard \cite{PG}, which gives, up to extractions, the structure of a bounded sequence of~$\dot{H}^s$, with~$s$ between~$0$ and~$\ds{\frac{3}{2}} \cdotp$ More precisely, the defect of compactness in the critical Sobolev embedding~$\ds{\dot{H}^s \subset L^p}$ is described in terms of a sum of rescaled and translated orthogonal profiles, up to a small term in~$L^p$. For more details about the history of the profile theory, we refer the reader to the paper \cite{P}.  

\medbreak
\begin{theorem} \sl{(Profile Theorem \cite{PG})\\ 
\label{theo profiles}
Let $(u_{0,n})_{n \in \N}$ be a bounded sequence  in $\dot{H}^s$. Then, up to an extraction:\\ 
- There exists a sequence vectors fields, called profiles $(\varphi^{j})_{j \in \N}$ in $\dot{H}^s$.\\ 
- There exists a sequence of scales and cores $(\lambda_{n,j},x_{n,j})_{n,j \in \N}$, such that, up to an extraction
$$\forall J \geqslant 0,\,\, u_{0,n}(x) = \sum_{j=0}^{J} \Lambda^{\frac{3}{p}}_{\lambda_{n,j},x_{n,j}}\varphi^{j}(x) + \psi_{n}^{J}(x) \quad \hbox{with}\quad \lim_{J \to +\infty}\limsup_{n \to+\infty}\|\psi_{n}^{J}\|_{L^{p}} =0, \quad \hbox{and} \quad p =\frac{6}{3-2s}\cdotp$$
Where, $(\lambda_{n,j},x_{n,j})_{n \in \N,j \in \N^*}$ are sequences of $(\R_{+}^* \times \R^3)^{\N}$ with the following orthogonality property:  for every integers $(j,k)$ such that $j \neq k$, we have
$$ \hbox{either}\lim_{n \to +\infty}\Bigl(\frac{\lambda_{n,j}}{\lambda_{n,k}} + \frac{\lambda_{n,k}}{\lambda_{n,j}}\Bigr) = +\infty  \quad\hbox {or} \quad  \lambda_{n,j} = \lambda_{n,k} \quad\hbox {and} \quad \lim_{n \to +\infty}\frac{|x_{n,j} - x_{n,k}|}{\lambda_{n,j}} = +\infty.$$
Moreover, for any $J \in \N$, we have the following orthogonality property
\begin{equation}
\label{ortho de la norme}
\| u_{0,n} \|^2_{\dot{H}^s} = \sum_{j=0}^{J} \| \varphi^{j} \|^2_{\dot{H}^s} + \| \psi_{n}^J \|^2_{\dot{H}^s} + \circ(1), \quad\hbox {when} \quad n \to +\infty.
\end{equation}
}\end{theorem}

\vskip 0.3cm
\noindent Let us recall a structure Lemma, based on the crucial profils theorem of P. G\'erard (see \cite{PG}). Let $(u_{0,n})$ be a bounded sequence in the Sobolev space $\dot{H}^s$, which profile decomposition is given by
$$ u_{0,n}(x) =  \sum_{j \in J}  \Lambda^{\frac{3}{p}}_{\lambda_{n,j},x_{n,j}} \varphi^{j}(x) + \psi_{n}^{J}(x),$$
with the appropriate properties on the error term ~$\psi_{n}^{J}$. By vertue of orthogonality of scales and cores given by Theorem \ref{theo profiles}, we sort profiles according to their scales
\begin{equation}
\label{decomposition 1}
 \begin{split}
 u_{0,n}(x) &= \sum_{\stackrel{j \in \mathcal{J}_{1}}{j \leqslant J}} \varphi^{j}(x-x_{n,j})
           + \sum_{\stackrel{j \in \mathcal{J}^{{}{c}}_{1}}{j \leqslant J}} \Lambda^{\frac{3}{p}}_{\lambda_{n,j},x_{n,j}}\varphi^{j}(x) + \psi_{n}^{J}(x)\\
 \end{split}
\end{equation}
\\
where 
for any $j \in \mathcal{J}_{1}$, for any $n \in \N$,\,  $\lambda_{n,j}  \equiv 1$.

\noindent Under these notations, we claim we have the following structure Lemma of the Navier-Stokes solutions, which proof will be provided in Section $8$.
\begin{lemma} \sl{(Profile decomposition of a sequence of Navier-Stokes solutions)\\
\label{lemme allure de la solution}
Let $(u_{0,n})_{n \geqslant 0}$ be a bounded sequence of initial data in $\dot{H}^s$ which profile decomposition is given by
$$u_{0,n}(x) = \sum_{j=0}^{J} \Lambda^{\frac{3}{p}}_{\lambda_{n,j},x_{n,j}}\varphi^{j}(x) + \psi_{n}^{J}(x).$$
\noindent Then, $\ds{\liminf_{n \geqslant 0}\ts(u_{0,n}) \geqslant \widetilde{T} \eqdefa \ds{\inf_{j \in \mathcal{J}_{1}}{\ts(\varphi^{j})}} }$ and for any $t < \ts(u_{0,n}) $, we have 
\begin{equation}
\label{decomposition de la solution}
\begin{split}  
NS(u_{0,n})(t,x) &= \sum_{j \in \mathcal{J}_{1}} NS(\varphi^{j})(t,x-x_{n,j})\, +\, e^{t\Delta} \Bigl(\sum_{\stackrel{j \in \mathcal{J}^{{}{c}}_{1}}{j \leqslant J}} \Lambda^{\frac{3}{p}}_{\lambda_{n,j},x_{n,j}}\varphi^{j}(x) + \psi_{n}^{J}(x) \Bigr)\,+\,   R_{n}^{J}(t,x)
\end{split} \end{equation}           
where the remaining term $R_{n}^{J}$ satisfies for any $T < \tilde{T}$, $\ds{\lim_{J \to +\infty}\lim_{n \to +\infty} \| R_{n}^{J} \|_{X^{s}_{T}}=0}$.\\
\noindent Moreover, we have the orthogonality property on the $\dot{H}^s$-norm for any $t < \tilde{T}$ 
\begin{equation}
\label{Pythagore}
\begin{split}
 \| NS(u_{0,n})(t)  \|^{2}_{\dot{H}^s} &= \sum_{j \in \mathcal{J}_{1}} \| NS(\varphi^{j})(t) \|^{2}_{\dot{H}^s} \,+ \, \Bigl\|  e^{t\Delta} \Bigl(\sum_{\stackrel{j \in \mathcal{J}^{{}{c}}_{1}}{j \leqslant J}} \Lambda^{\frac{3}{p}}_{\lambda_{n,j},x_{n,j}}\varphi^{j} + \psi_{n}^{J} \Bigr) \Bigr\|^{2}_{\dot{H}^s} +  \gamma_{n}^{J}(t).
\end{split}
\end{equation}
\noindent with $\ds{\lim_{J \to +\infty} \limsup_{n \to +\infty} \sup_{t' < t} | \gamma_{n}^{J}(t')|  = 0}$.
}\end{lemma} 

\medbreak
\noindent For the convenience of the reader, we recall the usual definition of Besov spaces. We refer the reader to~\cite{BCD}, from page~$63$, for a detailed presentation of the theory and analysis of homogeneous Besov spaces. 
\begin{definition}\sl{
 Let $s$ be in  $\R$, $(p,r)$ in $[1,+\infty]^2$ and $u$ in $\mathcal{S'}$. A tempered distribution $u$ is an element of the Besov space $\dot{B}^{s}_{p,r}$ if $u$ satifies $\ds{\lim_{j \to \infty}\, || \dot{S}_{j} u ||_{L^\infty} = 0 }$ and 
 $$ \| u\|_{\dot{B}^{s}_{p,r}} \eqdefa \Bigl(\sum_{j \in \Z} 2^{jrs}\,\,|| \dot{\Delta}_{j} u ||^{r}_{L^p}\Bigr)^{\frac{1}{r}}  < \infty,$$
 where $\dot{\Delta}_{j}$ is a frequencies localization operator (called Littlewood-Paley operator), defined by 
 $$ \dot{\Delta}_{j}u(\xi) \eqdefa \mathcal{F}^{-1}\bigl(\varphi(2^{-j}|\xi|)\widehat{u}(\xi)\bigr),$$
 with $\varphi \in \mathcal{D}([\frac{1}{2},2])$, such that $\ds{\sum_{j \in \Z} \varphi(2^{-j}t) = 1}$, for any $t>0$.
}\end{definition} 
\vsd
\begin{remark}
\label{equivalence norm besov}
\noindent Notice that the characterization of Besov spaces with positive indices in terms of finite differences is equivalent to the above definition (cf \cite{BCD}). In the case where the regularity index is between $0$ and $1$,  one has the following property. Let $s$ be in $]0,1[$ and $(p,r)$ in $[1,\infty]^2$. A constant $C$ exists such that, for any $u \in \mathcal{S{'}}$, 
\begin{equation}
C^{-1}\, \| u\|_{\dot{B}^{s}_{p,r}} \leqslant \Bigl\|  \frac{\| u(\cdotp -y) - u\|_{L^{p}}}{|y| ^s} \Bigr\|_{L^r(\R^d ; \frac{dy}{|y| ^d})} \, \leqslant C\,  \| u\|_{\dot{B}^{s}_{p,r}}.
\end{equation} 
\end{remark}
\medskip
\begin{remark}
Notice that $\dot{H}^s  \subset \dot{B}^{s}_{2,2}$ and both spaces coincide if $\ds{s < \frac{3}{2}\cdotp}$
\end{remark}

\medskip

\noindent We recall an interpolation property in Besov spaces, which will be useful in the sequel.  
\begin{prop}\sl{
\label{interpolation}
A constant $C$ exists which satisifes the following property. If $s_{1}$ and $s_{2}$ are real numbers such that $s_{1} < s_{2}$ and $\theta \in ]0,1[$, then we have for any $p\in [1,+\infty]$\\
\begin{equation*}
\| u\|_{\dot{B}^{\theta \,s_{1}  + (1-\theta)\, s_{2} }_{p,1}} \,\, \leqslant\,\,  C(s_{1},s_{2},\theta)\,  \| u\|^{\theta}_{\dot{B}^{s_{1}}_{p,\infty}}  \,\,  \| u\|^{1-\theta}_{\dot{B}^{s_{2}}_{p,\infty}}.
\end{equation*}
}\end{prop}

\medskip

 \section{Application of profile theory to sup-critical solutions}
\noindent This section is devoted to the proof of Lemma \ref{general lemma for critical element }. The statement given below is actually a bit stronger and clearly entails Lemma \ref{general lemma for critical element }. We shall prove the following proposition.  
\medskip
\begin{prop}\sl{
\label{proposition critical element }
Let $(v_{0,n})_{n \in \N}$ be a bounded sequence in $\dot{H}^s$ such that
\begin{equation*}
\tau^{*}(v_{0,n}) = 1 \quad \hbox{ and} \quad \hbox{ for any} \,\,\, \tau < 1, \quad (1-\tau)\, \| NS(v_{0,n})(\tau,) \|^{\sigma_{s}}_{\dot{H}^s} \leqslant\,\, M^{\sigma_{s}}_{c} + \varepsilon_{n},
\end{equation*}
where $\varepsilon_{n}$ is a generic sequence which tends to $0$ when $n$ goes to $+\infty$.\\
Then, up to extractions, we get the statements below\\
$\bullet$ the profile decomposition of such a sequence of data has a unique profile $\varphi^{j_{0}} $  with constant scale such that $NS(\varphi^{j_{0}} )$ is a sup-critical solution which blows up at time $1$, e.g 
\begin{equation}
\label{prop point 1}
\limsup_{\tau \to 1} (1 - \tau)\,  \| NS(\varphi^{j_{0}})(\tau) \|^{\sigma_{s}}_{\dot{H}^s} =\, M^{\sigma_{s}}_{c}.\\
\end{equation}
$\bullet$ "The limsup is actually a sup" 
\begin{equation}
\label{prop point 2}
\sup_{\tau < 1} (1 - \tau)\,  \| NS(\varphi^{j_{0}})(\tau) \|^{\sigma_{s}}_{\dot{H}^s} \,  =\, M^{\sigma_{s}}_{c}.\\
\end{equation}
}
\end{prop} 

\begin{proof}
\noindent Let $(v_{0,n})_{n \geqslant 1}$ be a bounded sequence in $\dot{H}^s $, satisfiying the assumptions of Proposition \ref{proposition critical element }. Therefore,  $(v_{0,n})_{n \geqslant 1}$ has the profile decomposition below
\begin{equation}
 \begin{split}
 v_{0,n}(x) &= \sum_{\stackrel{j \in \mathcal{J}_{1}}{j \leqslant J}} \varphi^{j}(x-x_{n,j})
           + \sum_{\stackrel{j \in \mathcal{J}^{{}{c}}_{1}}{j \leqslant J}} \Lambda^{\frac{3}{p}}_{\lambda_{n,j},x_{n,j}}\varphi^{j}(x) + \psi_{n}^{J}(x).
 \end{split}
\end{equation}
We denote by $\ds{ \tau^{*}_{j_{0}} \eqdefa \inf_{j \in \mathcal{J}_{1}}{\ts(\varphi^{j})}}$.\\

$\bullet$ Step $1$ : we start by proving by a contradiction argument that $\tau^{*}_{j_{0}} =1$. \\ \\ We have already known by vertue of Lemma \ref{lemme allure de la solution}, that $\tau^{*}_{j_{0}} \leqslant1$. Assuming that $\tau^{*}_{j_{0}} <1$, we expect a contradiction. Moreover, orthogonal Estimate (\ref{Pythagore}) can be bounded from below by
\begin{equation}
\label{equation de reference 2}
\begin{split}
\| NS(v_{0,n})(\tau) \|^2_{\dot{H}^s} \geqslant \| NS(\varphi^{j_{0}})(\tau) \|^2_{\dot{H}^s}  - |\gamma_{n}^{J}(\tau)|.
\end{split} 
\end{equation}
\noindent On the one hand, it seems clear by assumption that for any $\tau < \tau^{*}_{j_{0}}$, we have 
$$ (1 -\tau^{*}_{j_{0}})^{\frac{2}{\sigma_{s}}} \leqslant (1-\tau)^{\frac{2}{\sigma_{s}}}.$$
On the other hand, hypothesis on $NS(v_{0,n})$ yields 
$$ (1-\tau)^{\frac{2}{\sigma_{s}}} \, \| NS(v_{0,n})(\tau) \|^{2}_{\dot{H}^s} \leqslant\,\, M^{2}_{c} + \varepsilon_{n}.$$
Therefore, from the above remarks, we get 
\begin{equation}
\| NS(v_{0,n})(\tau) \|^{2}_{\dot{H}^s} \leqslant\,\, \frac{M^{2}_{c} + \varepsilon_{n}}{ (1 -\tau^{*}_{j_{0}})^{\frac{2}{\sigma_{s}}}} \cdotp
\end{equation}
\noindent Combining the above estimate with (\ref{equation de reference 2}), we finally get, after multiplication by the factor $(\tau^{*}_{j_{0}} - \tau)^{\frac{2}{\sigma_{s}}} $,
\begin{equation}
\label{equation de reference 3}
\begin{split}
\frac{M^{2}_{c} + \varepsilon_{n}}{ (1 - \tau^{*}_{j_{0}})^{\frac{2}{\sigma_{s}}}} \,\, (\tau^{*}_{j_{0}} - \tau)^{\frac{2}{\sigma_{s}}} \geqslant (\tau^{*}_{j_{0}} - \tau)^{\frac{2}{\sigma_{s}}} \, \| NS(\varphi^{j_{0}})(\tau) \|^2_{\dot{H}^s}  - (\tau^{*}_{j_{0}} - \tau)^{\frac{2}{\sigma_{s}}}\,  |\gamma_{n}^{J}(\tau)|.
\end{split} 
\end{equation}
\noindent Notice that $(\tau^{*}_{j_{0}} - \tau)^{\frac{2}{\sigma_{s}}}$ is always less than $1$, which allows us to get rid of it in front of the remaining term $|\gamma_{n}^{J}(\tau)|$. In addition, applying (\ref{remark c}) and hypothesis on the sequence $\varepsilon_{n}$, one has
\begin{equation*}
\begin{split}
\frac{M^{2}_{c} + \varepsilon_{0}}{ (1 -\tau^{*}_{j_{0}})^{\frac{2}{\sigma_{s}}}} \,\, (\tau^{*}_{j_{0}} - \tau)^{\frac{2}{\sigma_{s}}} \geqslant c\, - \, |\gamma_{n}^{J}(\tau)|.
\end{split} 
\end{equation*}
\noindent We first choose $\tau = \tau_{c}$ such that $\tau_{c} < \tau^{*}_{j_{0}}$ and $\ds{ \frac{M^{2}_{c} + \varepsilon_{0}}{ (1 -\tau^{*}_{j_{0}})^{\frac{2}{\sigma_{s}}}} \,\, (\tau^{*}_{j_{0}} - \tau_{c})^{\frac{2}{\sigma_{s}}} = \frac{c}{4}}\cdotp$ Then, we take $J$ and $n$ large enough such that $\ds{ |\gamma_{n}^{J}(\tau_{c})| \leqslant \frac{c}{2}}\cdotp$ Therefore,  we get a contradiction, which proves that $\tau^{*}_{j_{0}} = 1$.

$\bullet$ Step $2$ : we prove here that $ NS(\varphi^{j_{0}})$ is a sup-critical solution in $\dot{H}^s$. \\ \\ 
\noindent Let us come back to Inequality (\ref{equation de reference 2}), which we multiply by the factor $(1 -\tau)^{\frac{2}{\sigma_{s}}}$. As we have shown that $\tau^{*}_{j_{0}} = 1$, hypothesis on $NS(v_{0,n})$ implies that for any $\tau < 1$,
\begin{equation}
M^{2}_{c} + \varepsilon_{n}  \geqslant (1 - \tau)^{\frac{2}{\sigma_{s}}} \, \| NS(\varphi^{j_{0}})(\tau) \|^2_{\dot{H}^s}\,  -\,   |\gamma_{n}^{J}(\tau)|.
\end{equation}
\noindent Our aim is to prove that the particular profile $\varphi^{j_{0}}$ generates a sup-critical solution. If not, it means that 
$$ \exists \alpha_{0} >0, \forall \varepsilon >0, \,\, \exists \tau_{\varepsilon},\,\, \hbox{such that} \,\,0<  ( 1 -\tau_{\varepsilon} )^{\frac{2}{\sigma_{s}}}\, < \varepsilon \quad  \hbox{and} \quad  (1 -\tau_{\varepsilon} )^{\frac{2}{\sigma_{s}}}\,  \| NS(u_{0,n})(\tau_{\varepsilon}) \|^{2}_{\dot{H}^s} \geqslant M_{c}^2 + \alpha_{0}.$$
\noindent Taking the above inequality at time $\tau_{\varepsilon}$, one has
\begin{equation*}
M^{2}_{c} + \varepsilon_{n}  \geqslant\,  M_{c}^2 + \alpha_{0}  -\,   |\gamma_{n}^{J}(\tau_{\varepsilon})|.
\end{equation*}

\noindent Moreover, assumption on the remaining term $\gamma_{n}^{J}$ implies that
$$\forall \eta >0,\,\, \exists \widetilde{J}(\eta) \in \N, \,\, \exists N_{\eta} \in \N\,\, \hbox{such that}\,\, \forall J \geq \widetilde{J}(\eta),\,\, \forall n \geqslant N_{\eta},\,\,  |\gamma_{n}^{J}(\tau_{\varepsilon})|\leqslant \eta.$$

\noindent Let $\eta >0$. For any $J \geq \widetilde{J}(\eta)$ and for any $n \geqslant N_{\eta}\,\, $, we get at time $\tau_{\varepsilon}$,
\begin{equation*}
M^{2}_{c}  \geqslant  M_{c}^2 + \alpha_{0} \, - \, \eta.
\end{equation*}
\noindent Now, choosing $\eta$ small enough (namely $\ds{\eta = \frac{\alpha_{0}}{2}}$) we get a contradiction which proves that  $NS(\varphi^{j_{0}})$ is a sup-critical solution. This concludes the proof of step $2$ and thus the point (\ref{prop point 1}) is proved. \\

\medskip
$\bullet$ Step $3$ : let us prove the point (\ref{prop point 2}) of Proposition \ref{proposition critical element }. The proof is a straightforward adaptation of the previous one. We shall use that $NS(\varphi^{j_{0}})$ is a sup-critical solution: 
$$ \limsup_{\tau \to 1} (1 - \tau)\,  \| NS(\varphi^{j_{0}})(\tau) \|^{\sigma_{s}}_{\dot{H}^s} =\, M^{\sigma_{s}}_{c}.$$
\noindent As we always have 
$\ds{\sup_{\tau < 1}\, (1 - \tau)\,  \| NS(\varphi^{j_{0}})(\tau) \|^{\sigma_{s}}_{\dot{H}^s}   \geqslant \limsup_{\tau \to 1} (1 - \tau)\,  \| NS(\varphi^{j_{0}})(\tau) \|^{\sigma_{s}}_{\dot{H}^s}}$,
we get a first inequality : $\ds{\sup_{\tau < 1}\, (1 - \tau)\,  \| NS(\varphi^{j_{0}})(\tau) \|^{\sigma_{s}}_{\dot{H}^s}   \geqslant M^{\sigma_{s}}_{c} }$. \\
\noindent According to the previous computations, we have, for any $\tau < 1$,
\begin{equation*}
M^{2}_{c} + \varepsilon_{n}  \geqslant (1 - \tau)^{\frac{2}{\sigma_{s}}} \, \| NS(\varphi^{j_{0}})(\tau) \|^2_{\dot{H}^s}\,  -\,   |\gamma_{n}^{J}(\tau)|.
\end{equation*}
\noindent Hypothesis on the remaining term $|\gamma_{n}^{J}|$ implies that $\ds{\sup_{\tau < 1}\, (1 - \tau)\,  \| NS(\varphi^{j_{0}})(\tau) \|^{\sigma_{s}}_{\dot{H}^s}   \leqslant M^{\sigma_{s}}_{c} },$
which provides the second desired inequality. This ends up the proof of (\ref{prop point 2}).\\

\medskip
\noindent Let us recall some notation and add a few words about profiles with constant scale. Thanks to Lemma~\ref{lemme allure de la solution} and obvious boundaries from below we get for any $\ds{\tau < \ds{ \tau^{*}_{j_{0}} \eqdefa \inf_{j \in \mathcal{J}_{1}}{\ts(\varphi^{j})}} =1 }$
\begin{equation}
\begin{split}
\| NS(v_{0,n})(\tau) \|^2_{\dot{H}^s} \geqslant \sum_{j \in \mathcal{J}_{1}} \| NS(\varphi^{j})(\tau) \|^{2}_{\dot{H}^s} - |\gamma_{n}^{J}(\tau)|.\\ 
\end{split} 
\end{equation}
\noindent Among profiles with a scale equal to $1$ (e.g $j \in \mathcal{J}_{1}$), we distinguish profiles with a lifespan equal to~$\tau^{*}_{j_{0}} =1$ and profiles with a lifespan~$\tau^{*}_{j}$ strictly greater than $1$. In other words, we consider the set $$\tilde{\mathcal{J}_{1}} \eqdefa \{ j \in \mathcal{J}_{1} \, \, | \,\, \tau^{*}_{j} = 1 \}.$$
Therefore, for any $\tau < 1$,
\begin{equation*}
\begin{split}
\| NS(v_{0,n})(\tau) \|^2_{\dot{H}^s} &\geqslant \| NS(\varphi^{j_{0}})(\tau) \|^2_{\dot{H}^s} + \sum_{j \in \tilde{\mathcal{J}_{1}}, \, j \neq j_{0}} \, \| NS(\varphi^{j})(\tau) \|^{2}_{\dot{H}^s}\\  &+ \sum_{j \in \mathcal{J}_{1} \setminus \tilde{\mathcal{J}_{1}}} \,\| NS(\varphi^{j})(\tau) \|^{2}_{\dot{H}^s}\, - |\gamma_{n}^{J}(\tau)|, 
\end{split} 
\end{equation*}
\noindent which be bounded from below once again by
\begin{equation}
\label{equation de reference}
\begin{split}
\| NS(v_{0,n})(\tau) \|^2_{\dot{H}^s} \geqslant \| NS(\varphi^{j_{0}})(\tau) \|^2_{\dot{H}^s} + \sum_{j \in \tilde{\mathcal{J}_{1}}, \, j \neq j_{0}} \, \| NS(\varphi^{j})(\tau) \|^{2}_{\dot{H}^s}\, - |\gamma_{n}^{J}(\tau)|,
\end{split} 
\end{equation}
\noindent since obviously the term $\ds{\sum_{j \in \mathcal{J}_{1} \setminus \tilde{\mathcal{J}_{1}}} \,\| NS(\varphi^{j})(\tau) \|^{2}_{\dot{H}^s}\,}$ is positive.\\ \\

$\bullet$ Step $4$ : in order to complete the proof of Lemma \ref{general lemma for critical element }, we have to prove that there exists a unique profile with a lifespan~$\tau^{*}_{j_{0}} =1$, namely~$|\tilde{\mathcal{J}_{1}} | = 1$. Once again, we assume that there exists at least two profiles in $\tilde{\mathcal{J}_{1}}$. We expect a contraction. Arguments of the proof are similar to the ones used in the step~$2$. We shall use the fact~$\ds{(1-\tau)^{\frac{2}{\sigma_{s}}}\, \| NS(\varphi^{j})(\tau) \|^{2}_{\dot{H}^s}}$ can not be small as we want, by vertue of (\ref{remark c}). Indeed, let us come back to Inequality (\ref{equation de reference}). We have already proved that~$ \varphi^{j_{0}}$ generates a sup-critical solution, blowing up at time $1$. It means that for any~$\varepsilon >0$, there exists a time~$\tau_{\varepsilon}$ such that 
$$ 0<  ( 1 -\tau_{\varepsilon} )^{\frac{2}{\sigma_{s}}}\, < \varepsilon \quad  \hbox{and} \quad M_{c}^2 - \varepsilon   \leqslant(1 -\tau_{\varepsilon} )^{\frac{2}{\sigma_{s}}}\,  \| NS(\varphi^{j_{0}})(\tau_{\varepsilon}) \|^{2}_{\dot{H}^s} \leqslant M_{c}^2 + \varepsilon.$$

\noindent Therefore, Inequality (\ref{equation de reference}) becomes  at time $\tau_{\varepsilon} $
\begin{equation}
M^2_{c} + \varepsilon_{n}  \geqslant  M_{c}^2 - \varepsilon \, +\, \sum_{j \in \tilde{\mathcal{J}_{1}}, j \neq j_{0}} (1-\tau_{\varepsilon})^{\frac{2}{\sigma_{s}}}\, \| NS(\varphi^{j})(\tau_{\varepsilon}) \|^{2}_{\dot{H}^s} \, -\,   |\gamma_{n}^{J}(\tau_{\varepsilon})|.
\end{equation}
 
\noindent By vertue of (\ref{remark c}), there exists a universal constant $c >0$ such that for any $j \in \tilde{\mathcal{J}_{1}}$ and $j \neq j_{0}$
\begin{equation}
(1-\tau)^{\frac{2}{\sigma_{s}}}\, \| NS(\varphi^{j})(\tau) \|^{2}_{\dot{H}^s} \geqslant c^2. 
\end{equation}
\noindent As a result, taking the limit for $n$ and $J$ large enough, we infer that (still under the hypothesis ~$\ds{|\tilde{\mathcal{J}_{1}}| > 1}$)
\begin{equation}
M^{2}_{c}  \geqslant M_{c}^2 - \varepsilon \, +  (|\tilde{\mathcal{J}_{1}}|-1)\,c^2 - \eta.
\end{equation} 
\noindent Choosing $\varepsilon$ small enough, we get a contradiction and as a consequence, $\ds{|\tilde{\mathcal{J}_{1}}| = 1}$. It means there exists a unique profile generating a sub-critical solution, blowing up at time $1$. This completes the proof of Proposition \ref{proposition critical element }, and thus the proof of Lemma ~\ref{general lemma for critical element }.
\end{proof}

\medskip

\section{Fluctuation estimates in Besov spaces}

\vskip 0.3cm
\noindent This section is devoted to the proof of Lemma \ref{fluctuation lemma}. We shall prove some estimates on the fluctuation part which is given by the bilinear form $$B(u,u)(t) \eqdefa NS(u_{0})(t) - e^{t\Delta}u_{0} = u - e^{t\Delta}u_{0}.$$  We distinguish the case $\dot{B}^{\frac{1}{2}}_{2,\infty}$ from the case $\dot{B}^{s'}_{2,\infty}$, even if proves ideas are similar : we cut-off according low and high frequencies in the following sense : 
$$(\ts-t)2^{2j} \leqslant 1 \quad  \hbox{and} \quad  (\ts-t)2^{2j} \geqslant 1.$$ 
\noindent Concerning high frequencies, we shall use the regularization effet of the Laplacian. Let us start by proving the critical part of Lemma \ref{fluctuation lemma}. 
\medbreak
\begin{lemma}
\label{fluctuation 1/2}\sl{
Let $\ds{\frac{1}{2} < s< \frac{3}{2}}$ and $u_{0} \in \dot{H}^s $. It exists a positive constant $C_{s}$ such that
$$\hbox{If} \quad \ts(u_{0}) < \infty \quad \hbox{and} \quad   M_{u} \eqdefa (\ts(u_{0}) - t)\,  \| NS(u_{0})(t) \|^{\sigma_{s}}_{\dot{H}^s} < \infty,$$ 
\noindent then, we have 
  $$ \| u - e^{t\Delta}u_{0} \|_{\dot{B}^{\frac{1}{2}}_{2,\infty}} < C_{s}\, M^2_{u}.$$
}\end{lemma}
\begin{proof}
Duhamel formula gives
\begin{equation}
 u - e^{t\Delta}u_{0} \eqdefa B(u,u) = - \int_{0}^{t} e^{(t-t')\Delta}\, \P(\dive(u \otimes u)\, dt'.
\end{equation}
By vertue of classsical estimates on the heat term (see for instance Lemma $2.4$ in \cite{BCD}),  we have 
\begin{equation}
 \| \Delta_{j}e^{t\Delta}\, a \|_{L^2} \leqslant C\, e^{-ct\, 2^{2j}}\, \| \Delta_{j} a \|_{L^2}.
\end{equation}
Therefore, the fluctuation part becomes
\begin{equation}
 \begin{split}
   \| \Delta_{j} B(u,u)(t) \|_{L^2} 
   &\lesssim \int_{0}^{t}\,  e^{-c(t-t')\, 2^{2j}}\, 2^{j}\,  \| \Delta_{j}  (u \otimes u)(t') \|_{L^2}\, dt'\\
   &\lesssim \int_{0}^{t}\,   e^{-c(t-t')\, 2^{2j}}\, 2^{j}\, 2^{-j(2s-\frac{3}{2})}\,  \| u \otimes u(t') \|_{\dot{B}^{2s-\frac{3}{2}}_{2,\infty}}\, dt'.
 \end{split}
\end{equation}
We infer thus, thanks to the product laws in Sobolev spaces
\begin{equation}
 \begin{split}
  2^{\frac{j}{2}}\,   \| \Delta_{j} B(u,u) (t)\|_{L^2} &\lesssim   \int_{0}^{t}\,   e^{-c(t-t')\, 2^{2j}}\, 2^{j(3-2s)}\, \| u(t')\|^{2}_{\dot{H}^s}\, dt'.
 \end{split}
\end{equation}
\noindent By hypothesis, we have supposed that 
$$  M^2_{u} \eqdefa  (\ts(u_{0}) - t)\,\| NS(u_{0})(t) \|^{\sigma_{s}}_{\dot{H}^s} < \infty .$$
As a result, 
\begin{equation}
 \begin{split}
  2^{\frac{j}{2}}\,   \| \Delta_{j} B(u,u)(t) \|_{L^2} &\leqslant C_{s}\,   \int_{0}^{t}\,  e^{-c(t-t')\, 2^{2j}}\, 2^{j(3-2s)}\, \frac{M^2_{u}}{ (\ts(u_{0})-t')^{\frac{2}{\sigma_{s}}}}\\
  &= \int_{0}^{t}\, 1_{\{(\ts(u_{0})-t')2^{2j} \leqslant 1\}}\,  e^{-c(t-t')\, 2^{2j}}\, 2^{j(3-2s)}\, \frac{M^2_{u}}{ (\ts(u_{0})-t')^{\frac{2}{\sigma_{s}}}} \, dt'\\ &\qquad + \, \int_{0}^{t}\, 1_{\{(\ts(u_{0})-t')2^{2j} \geqslant 1\}}\,  e^{-c(t-t')\, 2^{2j}}\, 2^{j(3-2s)}\, \frac{M^2_{u}}{ (\ts(u_{0})-t')^{\frac{2}{\sigma_{s}}}}\, dt'\cdotp
 \end{split}
\end{equation}
We apply Young inequality : in the first integral, we consider $L^{\infty} \star L^{1}$,  whereas in the second one, we consider $L^{1} \star L^{\infty}$ in order to use the regularization effect of the Laplacian.
\begin{equation}
 \begin{split}
  2^{\frac{j}{2}}\,   \| \Delta_{j} B(u,u) (t)\|_{L^2} &\leqslant C_{s}\, M^2_{u} \, \int_{\ts(u_{0})-2^{-2j}}^{\ts(u_{0})} \frac{2^{j(3-2s)}\,dt'}{ (\ts(u_{0})-t')^{\frac{2}{\sigma_{s}}}}  +  C_{s}\,  M^2_{u} \, \int_{0}^{t}\,   e^{-c(t-t')\, 2^{2j}}\, 2^{j(3-2s)}\, 2^{2j(s-\frac{1}{2})} \, dt'.
 \end{split}
\end{equation}
We recall that $\ds{\frac{2}{\sigma_{s}}\eqdefa s-\frac{1}{2}}$ and $\ds{s-\frac{1}{2} < 1}$. As a result, 
\begin{equation}
 \begin{split}
  2^{\frac{j}{2}}\,   \| \Delta_{j} B(u,u)(t) \|_{L^2} &\leqslant\, C_{s}\, M^2_{u} \, \Bigl(2^{j(2s-3)}\,   2^{j(3-2s)}\,  \,  + \, \frac{1}{2^{2j}}\, 2^{j(3-2s)}\, 2^{2j(s-\frac{1}{2})}  \Bigr) 
 \lesssim C_{s}\, M^2_{u}. 
 \end{split}
\end{equation}
This concludes the proof on the fluctuation estimate in the critical case. 
\end{proof}

\vskip 0.5cm
\noindent The statement given below is a bit more general than the one of Lemma \ref{fluctuation lemma}, which we deduce immediately by an interpoaltion argument (the same as given at the end of the proof of Theorem \ref{Big key theorem}). \\
\begin{lemma}
\label{fluctuation s'}\sl{ Let $\ds{\frac{1}{2} < s< \frac{3}{2}}$ and $u_{0} \in \dot{H}^s $. It exists a positive constant $C_{s}$ such that
$$\hbox{If} \quad \ts(u_{0}) < \infty \quad \hbox{and} \quad   M_{u} \eqdefa (\ts(u_{0}) - t)\,  \| NS(u_{0})(t) \|^{\sigma_{s}}_{\dot{H}^s} < \infty,$$ 
\noindent then, we have for any $\ds{s < s'< 2s-\frac{1}{2}}$
$$ (\ts(u_{0})-t)^{\frac{1}{2}(s'-\frac{1}{2})}\,\, \| u(t) - e^{t\Delta}u_{0} \|_{\dot{B}^{s'}_{2,\infty}} < \infty.$$}
\end{lemma}

\begin{proof}
\noindent Same arguments as above yield
\begin{equation}
 \begin{split}
   \| \Delta_{j} B(u,u)(t) \|_{L^2} &\lesssim \int_{0}^{t}\,   e^{-c(t-t')\, 2^{2j}}\, 2^{j}\, 2^{-j(2s-\frac{3}{2})}\,  \| u \otimes u(t') \|_{\dot{B}^{2s-\frac{3}{2}}_{2,\infty}}\, dt'.
 \end{split}
\end{equation}
Product laws in Sobolev spaces and hypothesis on $u$ imply
\begin{equation}
 \begin{split}
  2^{js'}\,   \| \Delta_{j} B(u,u)(t) \|_{L^2} &\lesssim   \int_{0}^{t}\,   e^{-c(t-t')\, 2^{2j}}\,  2^{j(\frac{5}{2}-2s+s')}\, \| u(t')\|^{2}_{\dot{H}^s}\, dt'\\
  &\lesssim \int_{0}^{t}\,   e^{-c(t-t')\, 2^{2j}}\,  2^{j(\frac{5}{2}-2s+s')}\, \frac{C}{(\ts(u_{0})-t')^{s-\frac{1}{2}}}\cdotp
\end{split}
\end{equation}
We split (the same cut off as before) according low and high frequencies. Concerning high frequencies, since $\ts(u_{0}) -t \leqslant \ts(u_{0})-t'$, we get 
\begin{equation}
 \begin{split}
  2^{js'}\,   \| \Delta_{j} B(u,u)(t)\, 1_{\{(\ts-t)2^{2j} \geqslant 1\}} \|_{L^2}\, &\lesssim   \int_{0}^{t}\,  e^{-c(t-t')\, 2^{2j}}\, 2^{j(\frac{5}{2}-2s+s')}\, \frac{C}{ (\ts(u_{0})-t)^{s-\frac{1}{2}}}\, dt'\\
  & \lesssim \,  2^{j(\frac{1}{2}-2s+s')}\, \frac{C}{ (\ts(u_{0})-t)^{s-\frac{1}{2}}}\cdotp
 \end{split}
\end{equation}
\noindent Choosing $s'$ such that $\ds{\frac{1}{2}-2s+s' < 0}$, we get 
\begin{equation*}
 \begin{split}
  2^{js'}\,   \| \Delta_{j} B(u,u)(t) \,   1_{\{(\ts-t)2^{2j} \geqslant 1\}} \|_{L^2}\,  & \lesssim \, C\, \frac{ (\ts(u_{0}) - t)^{\frac{1}{2}(-\frac{1}{2}+2s-s')}}{ (\ts(u_{0})-t)^{s-\frac{1}{2}}} \,= C\,\, (\ts(u_{0}) - t)^{-\frac{1}{2}(s'-\frac{1}{2})},
 \end{split}
\end{equation*}
 \noindent which yields the desired estimate, as far as high frequencies are concerned.\\
\noindent Concerning low frequencies, let us come back to the very beginning. 
\begin{equation}
 \begin{split}
  2^{js'}\,   \| \Delta_{j} B(u,u)(t) \, 1_{\{(\ts(u_{0})-t)2^{2j} \leqslant 1\}} \|_{L^2}\, & \lesssim \, 2^{j(s'-s)}\,  2^{js}\, \| \Delta_{j} B(u,u) \|_{L^2}\\
  &\lesssim 2^{j(s'-s)}\, \| u(t) - e^{t\Delta}u_{0} \|_{\dot{B}^{s}_{2,\infty}}.
 \end{split}
\end{equation}
\noindent As $\ds{ \| u(t) - e^{t\Delta}u_{0} \|_{\dot{B}^{s}_{2,\infty}} \leqslant \frac{C}{(\ts(u_{0}) - t)^{\frac{1}{2}(s-\frac{1}{2})}}}$, we infer that 
\begin{equation*}
 \begin{split}
  2^{js'}\,   \| \Delta_{j} B(u,u)(t) \, 1_{\{(\ts(u_{0})-t)2^{2j} \leqslant 1\}}\, \|_{L^2}\, &\lesssim 2^{j(s'-s)}\, \frac{C}{(\ts(u_{0}) - t)^{\frac{1}{2}(s-\frac{1}{2})}}\cdotp
 \end{split}
\end{equation*}
\noindent Hypothesis of low frequencies implies
\begin{equation*}
 \begin{split}
  2^{js'}\,   \| \Delta_{j} B(u,u)(t) \, 1_{\{(\ts(u_{0})-t)2^{2j} \leqslant 1\}}\, \|_{L^2}\, &\lesssim \, \frac{C}{(\ts(u_{0}) - t)^{\frac{1}{2}(s-\frac{1}{2}) + \frac{1}{2}(s'-s)} }\,  = \, \frac{C}{(\ts(u_{0}) - t)^{\frac{1}{2}(s'-\frac{1}{2})}}\cdotp
 \end{split}
\end{equation*}
\noindent which ends up the proof for low frequency part.
\noindent The proof of Lemma \ref{fluctuation s'} is thus complete. 
\end{proof}

\medbreak
\medskip

\section{Existence of sup-critical solutions bounded in $\dot{B}^{\frac{1}{2}}_{2,\infty}$}
\noindent This section is devoted to complete the proof of Theorem \ref{Big key theorem}, namely the part concerning the $\dot{B}^{\frac{1}{2}}_{2,\infty}$-norm of the sup-critical solutions. We have already built some sup-critical elements in the space $\dot{H}^s$. It turns out that, starting from this statement, we shall prove that data generating a sup-critical element are not only in $\dot{H}^s$, but also in some others spaces such as $\dot{B}^{\frac{1}{2}}_{2,\infty}\cap \dot{B}^{s'}_{2,\infty}$, with $s'$ satisfiying the condition given below, which stems from the proof of Lemma \ref{fluctuation lemma}. \\
The statement given below is actually a bit stronger than the one we want to prove, since we are going to catch some sup-critical solutions not only in  $\dot{B}^{\frac{1}{2}}_{2,\infty}$ (as claimed by Theorem \ref{Big key theorem}) but also in $\dot{B}^{s'}_{2,\infty}$. The main idea to get such information on the regularity is to focus on the fluctuation part which is more regular than the solution itself. Notice that, in all this section, we use regularity index~$s'$ satisfying 
$$ s < s' < 2s- \frac{1}{2}\cdotp$$
\begin{theorem}
\label{theorem section5}\sl{
 There exists a data $\Phi_{0} \in  \dot{B}^{\frac{1}{2}}_{2,\infty}\cap \dot{H}^s \cap \dot{B}^{s'}_{2,\infty}$, such that $\ts(\Phi_{0}) < \infty$ and
 $$\sup_{t < \ts(\Phi_{0}) }\, (\ts(\Phi_{0}) - t)\, \| NS(\Phi_{0})(t) \|^{\sigma_{s}}_{\dot{H}^s} \, =\,\limsup_{t \to \ts(\Phi_{0})} (\ts(\Phi_{0}) - t)\, \| NS(\Phi_{0}) \|^{\sigma_{s}}_{\dot{H}^s} = M^{\sigma_{s}}_{c},$$
 $$\hbox{and for any} \quad t<\ts(\Phi_{0}), \quad \| NS(\Phi_{0}) \|_{\dot{B}^{\frac{1}{2}}_{2,\infty}} < \infty.$$
}\end{theorem}
\medbreak 

\begin{proof}
 \noindent The idea of the proof is to start with the existence of sup-sup-critical elements in $\dot{H}^s$. Indeed, we have proved previously that there exists a data $\Psi_{0} \in \dot{H}^s$, such that $\Psi \eqdefa NS(\Psi_{0})$ is sup-critical. Therefore, by definition of $\limsup$, there exists a sequence $t_{n} \nearrow \ts(\Psi_{0})$ such that 
 $$ \lim_{n \to +\infty} (\ts(\Psi_{0}) - t_{n})\, \| NS(\Psi_{0})(t_{n}) \|^{\sigma_{s}}_{\dot{H}^s} = M^{\sigma_{s}}_{c}.$$
 \noindent Let us introduce as before the rescaled sequence 
 $$v_{0,n}(y) =  \bigl(\ts(\Psi_{0})-t_{n}\bigr)^{\frac{1}{2}}\, NS(\Psi_{0})(t_{n}, \bigl(\ts(\Psi_{0})-t_{n}\bigr)^{\frac{1}{2}}y).$$
\noindent Such a sequence generates a solution which keeps on living until the time $1$ and satisfies
\begin{equation}
\begin{split}
 \|v_{0,n} \|^{\sigma_{s}}_{\dot{H}^s} &= \bigl(\ts(\Psi_{0})-t_{n}\bigr)\, \|NS(\Psi_{0,n})(t_{n}) \|^{\sigma_{s}}_{\dot{H}^s}. \\
\end{split}
\end{equation}
\noindent In the sake of simplicity, we note 
$$ \tau_{n} \eqdefa \ts(\Psi_{0})-t_{n}.$$
\noindent Previous computations imply that $(v_{0,n})$ is a bounded sequence of $\dot{H}^s$. Now, inspired by the idea of  Y. Meyer (fluctuation-tendancy method, \cite{YM}),  we decomposed the sequence $(v_{0,n})$ into 
\begin{equation}
 v_{0,n}(y) \eqdefa v_{0,n}(y) -\tau_{n}^{\frac{1}{2}}\, e^{t_{n}\Delta}\Psi_{0}( \tau_{n}^{\frac{1}{2}}\,y) \,\,\, + \,\,\, \tau_{n}^{\frac{1}{2}}\, e^{t_{n}\Delta}\Psi_{0}(\tau_{n}^{\frac{1}{2}}\,y),
\end{equation}
where we have 
\begin{equation*}
 v_{0,n}(y) \eqdefa \tau_{n}^{\frac{1}{2}}\,  NS(\Psi_{0})(t_{n}, \tau_{n}^{\frac{1}{2}}\, y)
\end{equation*}
It follows
\begin{equation}
 v_{0,n}(y) \eqdefa \tau_{n}^{\frac{1}{2}}\, \underbrace{\Bigl( NS(\Psi_{0})(t_{n},\cdotp) -  e^{t_{n}\Delta}\Psi_{0}    \Bigr)}_{B(\Psi,\Psi)(t_{n}) = \hbox{fluctuation part}} (\tau_{n}^{\frac{1}{2}}\, y) \,\,\, + \,\,\, \tau_{n}^{\frac{1}{2}} \underbrace{e^{t_{n}\Delta}\Psi_{0}}_{\hbox{tendancy part}}(\tau_{n}^{\frac{1}{2}}\,y).
\end{equation}
\begin{lemma}
\label{fluctuation born\'ee dans 3 espaces}\sl{
 The rescaled fluctuation part $\phi_{n} \eqdefa \tau_{n}^{\frac{1}{2}}\, B(\Psi,\Psi)(t_{n}, \tau_{n}^{\frac{1}{2}}\, \cdotp)$ is bounded in $\dot{H}^s \cap \dot{B}^{\frac{1}{2}}_{2,\infty} \cap \dot{B}^{s'}_{2,\infty}$.
}\end{lemma}
\vsd
\begin{proof}
\noindent Indeed, concerning the $\dot{B}^{\frac{1}{2}}_{2,\infty}$-norm, we use firstly the scaling invariance of this norm and then we apply Lemma \ref{fluctuation lemma}, which gives
\begin{equation}
\sup_{n}\,  \| \phi_{n} \|_{\dot{B}^{\frac{1}{2}}_{2,\infty}} = \sup_{n}\, \|  NS(\Psi_{0})(t_{n},\cdotp) -  e^{t_{n}\Delta}\Psi_{0} \|_{\dot{B}^{\frac{1}{2}}_{2,\infty}} < \infty.
\end{equation}

\noindent Concerning the $\dot{H}^s$-norm, we apply successively the following arguments : scaling, triangular inequality and the fact that $NS(\Psi_{0})$ is a sup-critical element in $\dot{H}^s$. 
\begin{equation}
\begin{split}
 \| \phi_{n} \|^{\sigma_{s}}_{\dot{H}^{s}} &= \tau_{n} \, \|  NS(\Psi_{0})(t_{n},\cdotp) -  e^{t_{n}\Delta}\Psi_{0} \|^{\sigma_{s}}_{\dot{H}^{s}}\\
 &\lesssim \tau_{n} \, \|  NS(\Psi_{0})(t_{n},\cdotp) \|^{\sigma_{s}}_{\dot{H}^{s}} + \tau_{n} \, \|  e^{t_{n}\Delta}\Psi_{0} \|^{\sigma_{s}}_{\dot{H}^{s}}\\
 &\lesssim  \Bigl(M_{c} + \frac{1}{n}\Bigr)^{\sigma_{s}} + \tau_{n} \, \|  \Psi_{0} \|^{\sigma_{s}}_{\dot{H}^{s}} < \infty.\\
 \end{split}
\end{equation}
\noindent Therefore, $\ds{\sup_{n}\,   \| \phi_{n} \|^{\sigma_{s}}_{\dot{H}^{s}} < \infty}$.\\
\noindent Concerning the $\dot{B}^{s'}_{2,\infty}$-norm, scaling argument combinig with Lemma \ref{fluctuation lemma} yields
\begin{equation}
\begin{split}
 \| \phi_{n} \|^{\sigma_{s'}}_{\dot{B}^{s'}_{2,\infty}} &= \tau_{n} \, \|  NS(\Psi_{0})(t_{n},\cdotp) -  e^{t_{n}\Delta}\Psi_{0} \|^{\sigma_{s'}}_{\dot{B}^{s'}_{2,\infty}}.\\
 \end{split}
\end{equation}
\noindent This concludes the proof of this Lemma \ref{fluctuation born\'ee dans 3 espaces}.
\end{proof}

\noindent By vertue of profile theory, we perform a profile decomposition of the sequence $\phi_{n}$ in the Sobolev space ~$\dot{H}^s$. But in this decomposition, there is only left profiles with constant scale, as Lemma below will prove it. The idea is clear. As $\phi_{n}$ is bounded in the Besov space $\dot{H}^s \cap \dot{B}^{\frac{1}{2}}_{2,\infty}$, big scales vanish. Likewise, the fact that $\phi_{n}$ is bounded in the Besov space $\dot{H}^s \cap  \dot{B}^{s'}_{2,\infty}$  implies that small scales vanish. That is the point in the Lemma below.
\medskip 
\begin{lemma}
\label{petite et grande echelle}\sl{
 $\bullet$ If $(f_{n})$ is a bounded sequence in $\dot{B}^{\frac{1}{2}}_{2,\infty} \cap \dot{H}^s$ and if $\ds{\limsup_{n \to +\infty}\|f_{n}\|_{\dot{B}^{s}_{2,\infty}} = L >0}$, then there is no big scales in the profile decomposition of the sequence $f_{n}$ in $\dot{H}^s$. \\
 $\bullet$ If $(f_{n})$ is a bounded sequence in $\dot{B}^{s'}_{2,\infty} \cap \dot{H}^s$, with $\ds{s' >s>\frac{1}{2}}$ and if $\ds{\limsup_{n \to +\infty}\|f_{n}\|_{\dot{B}^{s}_{2,\infty}} = L >0}$, then there is no small scales in the profile decomposition of the sequence $f_{n}$ in $\dot{H}^s$. \\
}\end{lemma}
\medbreak
 \begin{proof}
 We only proof the first part of the Lemma. The other one is similar. If ~$\ds{\limsup_{n \to +\infty}\|f_{n}\|_{\dot{B}^{s}_{2,\infty}} = L >0}$, it means there exists an extraction $\varphi(n)$ such that $\ds{\| f_{\varphi(n)} \|_{\dot{B}^{s}_{2,\infty}} \geqslant \frac{L}{2}} \cdotp$ Otherwise, for any subsequence of $(f_{n})$, we would have 
\begin{equation*}
  \| f_{\varphi(n)} \|_{\dot{B}^{s}_{2,\infty}} < \frac{L}{2} \quad \hbox{and thus,} \quad \lim_{n \to +\infty} \| f_{\varphi(n)} \|_{\dot{B}^{s}_{2,\infty}} \leqslant \frac{L}{2}\cdotp
\end{equation*}
\noindent As a result, we would have $\ds{\limsup_{n \to +\infty}\|f_{n}\|_{\dot{B}^{s}_{2,\infty}} \leqslant \frac{L}{2} < L}$, which is wrong by hypothesis. Moreover, by definition of the Besov norm, we can find a sequence $(k_{n})_{n \in \Z}$, such that 
 \begin{equation}
  \lim_{n \to +\infty} 2^{k_{n}s} \| \Delta_{k_{n}}\, f_{\varphi(n)}\|_{L^2} = \| f_{\varphi(n)} \|_{\dot{B}^{s}_{2,\infty}}.
 \end{equation}
Therefore, $\ds{\lim_{n \to +\infty} 2^{k_{n}s} \| \Delta_{k_{n}}\, f_{\varphi(n)}\|_{L^2} \geqslant \frac{L}{2} }$. \\
Let us introduce the scale $\ds{\lambda_{n} \eqdefa 2^{-k_{n}}}$. As (up to extraction) \, $\ds{2^{k_{n}s} \| \Delta_{k_{n}}\, f_{\varphi(n)}\|_{L^2} \geqslant \frac{L}{2}}$, then one has  
\begin{equation*}
 2^{k_{n}(s-\frac{1}{2})}\,\, \| f_{\varphi(n)} \|_{\dot{B}^{\frac{1}{2}}_{2,\infty}}\geqslant \frac{L}{2}\cdotp 
\end{equation*}
Hence, the infimum limit of the sequence $k_{n}$ is not $-\infty$, otherwise, the term $\ds{ 2^{k_{n}(s-\frac{1}{2})}}$ would tend to $0$ and thus $L=0$ (since the sequence $\| f_{\varphi(n)} \|_{\dot{B}^{\frac{1}{2}}_{2,\infty}}$ is bounded by hypothesis), which is false by hypothesis. Therefore, $\lambda_{n} \nrightarrow +\infty$ : big scales are excluded from the profile decomposition of the sequence $f_{n}$. This concludes the proof of Lemma \ref{petite et grande echelle}.
 \end{proof}
\noindent \textit{Continuation of the proof of Theorem \ref{theorem section5}}. \\Let us come back to the proof of sup-critical element in the Besov space $\dot{B}^{\frac{1}{2}}_{2,\infty}\cap \dot{B}^{s'}_{2,\infty}$. Firstly, we check that $\phi_{n}$ satisfies hypothesis of Lemma \ref{petite et grande echelle}. As it was already checked previously, $\phi_{n}$ is bounded in $\dot{B}^{\frac{1}{2}}_{2,\infty} \cap \dot{H}^s \cap \dot{B}^{s'}_{2,\infty}$. Concerning assumption $\ds{\limsup_{n \to +\infty} \| \phi_{n}\|_{\dot{B}^{s}_{2,\infty}} >0}$, by scaling argument, one has 
\begin{equation}
\begin{split}
\|\phi_{n}\|^{\sigma_{s}}_{\dot{B}^{s}_{2,\infty}} &= \tau_{n} \, \|  NS(\Psi_{0})(t_{n},\cdotp) -  e^{t_{n}\Delta}\Psi_{0} \|^{\sigma_{s}}_{\dot{B}^{s}_{2,\infty}}
= (\ts(\Psi_{0}) - t_{n})    \|  NS(\Psi_{0})(t_{n},\cdotp) -  e^{t_{n}\Delta}\Psi_{0} \|^{\sigma_{s}}_{\dot{B}^{s}_{2,\infty}}\\ 
&\geqslant (\ts(\Psi_{0}) - t_{n})    \|  NS(\Psi_{0})(t_{n},\cdotp)\|^{\sigma_{s}}_{\dot{B}^{s}_{2,\infty}} \,\, -   \,\,(\ts(\Psi_{0}) - t_{n})  \|  \Psi_{0} \|^{\sigma_{s}}_{\dot{H}^{s}}.
\end{split}
\end{equation}
Obviously, the term $\ds{(\ts(\Psi_{0}) - t_{n})  \|  \Psi_{0} \|^{\sigma_{s}}_{\dot{H}^{s}}}$ tends to $0$ when $n$ goes to $+\infty$. By vertue of (\ref{remark c}) and \cite{LR}, there exists a constant $c>0$ such that $\ds{(\ts(\Psi_{0}) - t_{n})    \|  NS(\Psi_{0})(t_{n},\cdotp)\|^{\sigma_{s}}_{\dot{B}^{s}_{2,\infty}}  \geqslant c}$. Therefore, $$\limsup_{n \to +\infty} \|\phi_{n}\|_{\dot{B}^{s}_{2,\infty}} >0$$ 
\noindent and thus profile decomposition of  $\phi_{n}$ in the space $\dot{H}^s$ is reduced to (with notations of Theorem \ref{theo profiles})
\begin{equation}
 \phi_{n} = \sum_{j \geqslant 0}^{J} V^{j}(\cdotp-x_{n,j}) \, + \, r_{n}^{J}.
\end{equation}
Moreover, as the sequence $\phi_{n}$ is bounded in $\dot{B}^{\frac{1}{2}}_{2,\infty} \cap \dot{B}^{s'}_{2,\infty}$, profiles $V^{j}$  belong also to $\dot{B}^{\frac{1}{2}}_{2,\infty} \cap \dot{B}^{s'}_{2,\infty}$. That's the crucial point in the proof. Indeed, each profile $V^{j}$  can be seen as a translated (by $x_{n,j}$) weak limit of the sequence $\phi_{n}$. As a result, we get immediately
$$  \|    V^{j}   \|_{\dot{B}^{\frac{1}{2}}_{2,\infty}}  \leqslant  \liminf_{n \to +\infty} \|     \phi_{n} \|_{\dot{B}^{\frac{1}{2}}_{2,\infty}} < \infty \quad \hbox{and} \quad \|    V^{j}   \|_{\dot{B}^{s'}_{2,\infty}}  \leqslant  \liminf_{n \to +\infty} \|     \phi_{n} \|_{\dot{B}^{s'}_{2,\infty}} < \infty.$$


\noindent Let us come back to the sequence $(v_{0,n})$ defined by $$ v_{0,n} \eqdefa \phi_{n} \,\, +\,\,  \tau_{n}^{\frac{1}{2}}\ e^{t_{n}\Delta}\Psi_{0}(\tau_{n}^{\frac{1}{2}}\,\cdotp).$$ As it has been already underlined previously, the term $\ds{\gamma_{n} \eqdefa \tau_{n}^{\frac{1}{2}}\ e^{t_{n}\Delta}\Psi_{0}(\tau_{n}^{\frac{1}{2}}\,\cdotp)}$ tends to $0$ in $\dot{H}^s$-norm (and thus in $L^p$-norm, by Sobolev embedding) since 
\begin{equation}
 \| \tau_{n}^{\frac{1}{2}}\ e^{t_{n}\Delta}\Psi_{0}(\tau_{n}^{\frac{1}{2}}\,\cdotp) \|^{\sigma_{s}}_{\dot{H}^s} = \tau_{n} \, \| \ e^{t_{n}\Delta}\Psi_{0} \|^{\sigma_{s}}_{\dot{H}^s} \leqslant \tau_{n} \, \| \ \Psi_{0}\|^{\sigma_{s}}_{\dot{H}^s}.
\end{equation}
\noindent Combining the profile decomposition of $(\phi_{n})$ with the definition of $(v_{0,n})$, we finally get
\begin{equation*}
v_{0,n} = \sum_{j \geqslant 0}^{J} V^{j}(\cdotp-x_{n,j}) \, + \, r_{n}^{J} \, +\, \gamma_{n},
\end{equation*}
with $\ds{\lim_{J \to +\infty}\limsup_{n \to+\infty}\|r_{n}^{J}\|_{L^{p}} =0}$ and $\ds{ \lim_{n \to+\infty}\|  \gamma_{n}  \|_{L^{p}} =0}$.
\noindent By vertue of Lemma \ref{lemme allure de la solution}, one has for any $\tau <1$
\begin{equation*}
  NS(v_{0,n})(\tau) = \sum_{j \geqslant 0}^{J} NS(V^{j})(\tau,\cdotp-x_{n,j}) + e^{\tau \Delta} (r_{n}^{J} + \gamma_{n} ) + R_{n}^{J}(\tau).
\end{equation*}
\noindent By definition of the sequence $(v_{0,n})$, $ NS(v_{0,n})$ is given by 
$$  NS(v_{0,n})(\tau,\cdotp) = \bigl(\ts(\Psi_{0})-t_{n}\bigr)\,^\frac{1}{2} \, NS(\Psi_{0})\bigl( t_{n} + \tau\,\bigl(\ts(\Psi_{0})-t_{n}\bigr)\, ,\, \bigl(\ts(\Psi_{0})-t_{n}\bigr)\,^\frac{1}{2}\,\cdotp \bigr).$$
\noindent Once again, we denote $\ds{\widetilde{t_{n}} = t_{n} + \tau\,\bigl(\ts(\Psi_{0})-t_{n}\bigr)\,}$ and one has
\begin{equation*}
(1-\tau)\, \|NS(v_{0,n})(\tau,\cdotp) \|^{\sigma_{s}}_{\dot{H}^s} = \bigl(\ts(\Psi_{0}) - \widetilde{t_{n}}\bigr) \, \|NS(\Psi_{0})\bigl( \widetilde{t_{n}} ,\cdotp \bigr)\|^{\sigma_{s}}_{\dot{H}^s}.
\end{equation*}
\noindent As $\widetilde{t_{n}} \geqslant t_{n}$ for any $n$, we get
$$(1-\tau)\| NS(v_{0,n})(\tau) \|^{\sigma_{s}}_{\dot{H}^s} = (\ts(\Psi_{0})-\widetilde{t_{n}})\| NS(\Psi_{0})(\widetilde{t_{n}}) \|^{\sigma_{s}}_{\dot{H}^s} \leqslant M^{\sigma_{s}}_{c} + \frac{2}{n}\cdotp$$
\noindent Hence, Proposition \ref{proposition critical element } implies there exists some a unique profile $ \Phi_{0} $ in~$\dot{B}^{\frac{1}{2}}_{2,\infty}~\cap ~\dot{H}^s~\cap \dot{B}^{s'}_{2,\infty}$ such that the $NS$-solution genrated by this profile is a sup-critical solution. As $ \Phi_{0} $ belongs to~$\dot{B}^{\frac{1}{2}}_{2,\infty}$, Lemma \ref{fluctuation lemma} implies that $NS(\Phi_{0})$ is bounded in the same space.  This ends up the proof of Theorem \ref{theorem section5}.\\ Hence, we claim that the proof of Theorem \ref{Big key theorem} is over. Indeed, this stems from an interpolation argument. By vertue of Proposition \ref{interpolation}, we have for any $\ds{s < s_{1} < s'}$
\begin{equation}
\begin{split}
\| \Phi_{0} \|_{\dot{H}^{s_{1}}} \leqslant \| \Phi_{0} \|_{\dot{B}^{s_{1}}_{2,1}} \leqslant  \, \| \Phi_{0}  \|^{\theta}_{\dot{B}^{s}_{2,\infty}} \,\, \| \Phi_{0}  \|^{1-\theta}_{\dot{B}^{s'}_{2,\infty}} \,\, \leqslant  \, \| \Phi_{0}  \|^{\theta}_{\dot{H}^{s}}\,\, \| \Phi_{0}  \|^{1-\theta}_{\dot{B}^{s'}_{2,\infty}}.
\end{split}
\end{equation}
This concludes the proof of Theorem \ref{Big key theorem}. 
\end{proof}

\section{Another notion of critical solution}
\noindent In this section, we wonder if among sup-critical solutions,  we can find some of them which reach the biggest infimum limit of the quantity~$\ds{(\ts(u_{0}) - t)\, \| NS(u_{0})(t) \|^{\sigma_{s}}_{\dot{H}^s}}$. We define the following set $\mathcal{E}_{c}$ by
\begin{equation*}
\begin{split}
\mathcal{E}_{c} \eqdefa  &\Bigl\{ u_{0} \in \dot{B}^{\frac{1}{2}}_{2,\infty} \cap \dot{H}^s \cap \dot{B}^{s'}_{2,\infty}  \quad \hbox{such that} \,\, \ts(u_{0}) < \infty \,\, ;\\
&\sup_{t <\ts(u_{0})}\, (\ts(u_{0}) - t)\,  \| NS(u_{0})(t) \|^{\sigma_{s}}_{\dot{H}^s}  =\limsup_{t \to \ts(u_{0})} (\ts(u_{0}) - t)\,  \| NS(u_{0})(t) \|^{\sigma_{s}}_{\dot{H}^s} =\, M^{\sigma_{s}}_{c}\,\, ;\\
& \hbox{for any} \quad t<\ts(u_{0}),   \quad \| NS(u_{0})(t) \|_{\dot{B}^{\frac{1}{2}}_{2,\infty}} < \infty
\quad \hbox{and} \quad (\ts(u_{0}) - t)\,\,  \| NS(u_{0})(t) \|^{^{\sigma_{s'}}}_{\dot{B}^{s'}_{2,\infty}} < \infty \Bigr\}. 
\end{split}
\end{equation*}
\noindent Let us introduce the following quantity  $m^{\sigma_{s}}_{c}$ 
$$ m^{\sigma_{s}}_{c}  \eqdefa  \sup_{u_{0} \,\in\, \mathcal{E}_{c} }   \bigl\{  \liminf_{t \to \ts(u_{0})}(\ts(u_{0}) - t)\, \| NS(u_{0})(t) \|^{\sigma_{s}}_{\dot{H}^s} \bigr\}.$$
\begin{definition}\sl{(sup-inf-critical solution)\\
A solution $u = NS(u_{0})$ is said to be a sup-inf-critical solution if $u_{0}$ belongs to  $\mathcal{E}_{c}$ and 
\begin{equation}
\begin{split}
\liminf_{t \to \ts(u_{0})} (\ts(u_{0}) - t)\,  \| NS(u_{0})(t) \|^{\sigma_{s}}_{\dot{H}^s} =\, m^{\sigma_{s}}_{c}. 
\end{split}
\end{equation}
}\end{definition}
\noindent  Notice we need to look for such elements among sup-critical solutions, otherwise the definition of $m^{\sigma_{s}}_{c}$ would be meaningless. We claim that there exist such elements.

\begin{lemma}
\label{sup inf critical lemma}\sl{
There exists some elements belonging to $\mathcal{E}_{c}$, which are sup-inf-critical. 
}\end{lemma}

\begin{proof}
By definition of $m^{\sigma_{s}}_{c}$, we can find a sequence $(u_{0,n}) \in \dot{H}^s $ and a sequence $t_{n} \nearrow \ts(u_{0,n}) \equiv \ts $ (we can assume this, up to a rescaling) such that
\begin{equation}
\label{liminf tn}
m_{c} - \varepsilon_{n} \leqslant (\ts - t_{n})^{\frac{1}{\sigma_{s}}}\,  \| NS(u_{0,n})(t_{n}) \|_{\dot{H}^s} \leqslant m_{c} + \varepsilon_{n}
\end{equation} 
and 
\begin{equation}
\hbox{For any} \quad t \geqslant t_{n}, \quad   m_{c} - \varepsilon_{n} \leqslant (\ts - t)^{\frac{1}{\sigma_{s}}}\,  \| NS(u_{0,n})(t) \|_{\dot{H}^s}.
\end{equation} 
\noindent Assume in addition that the sequence $(u_{0,n})$ belongs to the set $\mathcal{E}_{c}$. As a consequence, we have 
\begin{equation}
\label{hypothese M_c}
\hbox{For any} \quad t \geqslant t_{n} \quad,   m_{c} - \varepsilon_{n} \leqslant (\ts - t)^{\frac{1}{\sigma_{s}}}\,  \| NS(u_{0,n})(t) \|_{\dot{H}^s} \leqslant M_{c} + \varepsilon_{n}.
\end{equation}
\noindent Considering the rescaled sequence $$v_{0,n}(y) = \bigl(\ts - t_{n}\bigr)^\frac{1}{2} \, NS(u_{0,n})\bigl( t_{n},(\ts - t_{n}\bigr)^\frac{1}{2}\,y \bigr).$$
Hence, $v_{0,n}$ satisfies properties below by scaling argument
\begin{equation}
\begin{split}
 \|v_{0,n} \|^{\sigma_{s}}_{\dot{H}^s} = \bigl(\ts -t_{n}\bigr)\,& \|NS(u_{0,n})(t_{n}) \|^{\sigma_{s}}_{\dot{H}^s}, \quad
 \|v_{0,n} \|_{\dot{B}^{\frac{1}{2}}_{2,\infty}} = \|NS(u_{0,n})(t_{n}) \|_{\dot{B}^{\frac{1}{2}}_{2,\infty}}\\
 &\hbox{and} \quad \|v_{0,n} \|^{\sigma_{s'}}_{\dot{B}^{s'}_{2,\infty}} = \bigl(\ts -t_{n}\bigr)\, \|NS(u_{0,n})(t_{n}) \|^{\sigma_{s'}}_{\dot{B}^{s'}_{2,\infty}}.  
\end{split}
\end{equation}
\noindent Combining (\ref{liminf tn}) with the fact that  $(u_{0,n})$ belongs to $\mathcal{E}_{c}$, we infer that the sequence $(v_{0,n})_{n \geqslant 1}$ is bounded in $ \dot{B}^{\frac{1}{2}}_{2,\infty} \cap \dot{H}^s \cap \dot{B}^{s'}_{2,\infty} $. 
Moreover, concerning the Navier-Stokes solution generated by such a data~$NS(v_{0,n}) $, we know that it keeps on living until the time~$\tau^* = 1$ and satisfies once again (with~$\ds{\widetilde{t_{n}} = t_{n} + \tau\,\bigl(\ts -t_{n}\bigr)\,}$)
\begin{equation}
(1-\tau)^{\frac{1}{\sigma_{s}}}\, \|NS(v_{0,n})(\tau) \|_{\dot{H}^s} = (\ts - \widetilde{t_{n}})^{\frac{1}{\sigma_{s}}}\, \|NS(u_{0,n})( \widetilde{t_{n}})\|_{\dot{H}^s}.
\end{equation}
\noindent As $\widetilde{t_{n}} \geqslant t_{n}$ for any $n$, we infer that for any $\tau <1$
$$(1-\tau)^{\frac{1}{\sigma_{s}}}\, \| NS(v_{0,n})(\tau) \|_{\dot{H}^s} \geqslant m_{c} - \varepsilon_{n}.$$
\noindent Let us sum up information we have on the sequence $v_{0,n}$. Firstly, the lifespan of the Navier-Stokes associated with the sequence $v_{0,n}$ is equal to $1$. Then, 
\begin{equation*}
\limsup_{\tau \to 1} (1-\tau)^{\frac{1}{\sigma_{s}}}\, \,\| NS(v_{0,n})(\tau) \|_{\dot{H}^s} = \limsup_{ \widetilde{t_{n}} \to \ts} \,\, (\ts-\widetilde{t_{n}})^{\frac{1}{\sigma_{s}}}\, \| NS(u_{0,n})(\widetilde{t_{n}}) \|_{\dot{H}^s},
\end{equation*}
which implies, thanks to (\ref{hypothese M_c}) and definition of $M_{c}$ , that for any $\tau <1$,
\begin{equation*}
\limsup_{\tau \to 1} \,\,(1-\tau)^{\frac{1}{\sigma_{s}}}\,\, \| NS(v_{0,n})(\tau) \|_{\dot{H}^s} = M_{c} \quad \hbox{and} \quad \| NS(v_{0,n})(\tau) \|_{\dot{B}^{\frac{1}{2}}_{2,\infty}} = \| NS(u_{0,n})(\widetilde{t_{n}}) \|_{\dot{B}^{\frac{1}{2}}_{2,\infty}} < \infty.
\end{equation*}
\noindent In addition, 
\begin{equation}
\begin{split} 
(1-\tau)^{\frac{1}{\sigma_{s'}}}\, \,\| NS(v_{0,n})(\tau) \|_{\dot{B}^{s'}_{2,\infty}} = (\ts-\widetilde{t_{n}})^{\frac{1}{\sigma_{s'}}}\, \| NS(u_{0,n})(\widetilde{t_{n}}) \|_{\dot{B}^{s'}_{2,\infty}} < \infty.
\end{split}
\end{equation}
\noindent To summerize, from the minimizing sequence $(u_{0,n})$ of the set $\mathcal{E}_{c} $, we build another sequence $(v_{0,n})$ (the rescaled sequence of $(u_{0,n})$) which also belongs to the set $\mathcal{E}_{c}$. Moreover, as the sequence $(v_{0,n})$  is bounded in the spaces $ \dot{B}^{\frac{1}{2}}_{2,\infty} \cap \dot{H}^s \cap \dot{B}^{s'}_{2,\infty} $ and satisfies $\ds{\limsup_{n \to +\infty}\,\,  \| v_{0,n} \|_{\dot{B}^{s}_{2,\infty}}    } < \infty$, Lemma \ref{petite et grande echelle} implies that  profile decomposition in $\dot{H}^s$ of such a sequence is reduced, up to extractions, to a sum of translated profiles and a remaining term (under notations of Theorem \ref{theo profiles})
$$ v_{0,n} =  \sum_{j \in \mathcal{J}_{1}} \varphi^{j}(\cdotp -x_{n,j}) + \psi_{n}^{J}.$$
\noindent By vertue of Theorem \ref{lemme allure de la solution}, combining with Proposition \ref{proposition critical element }, we infer there exists only one profile $\varphi^{j_{0}}$ which blows up at time $1$ and such that
\begin{equation}
NS(v_{0,n})(\tau,\cdotp) = NS(\varphi^{j_{0}})(\tau,\cdotp -x_{n,j_{0}}) \,\, +\, \sum_{\stackrel{j \in \mathcal{J}_{1}, j \neq j_{0}}{ \tau^{j}_{*} > 1}} NS(\varphi^{j})(\cdotp -x_{n,j}) +\,  \, e^{\tau \Delta}\psi_{n}^J (\cdotp) \,\, + \,\, R_{n}^J(\tau,\cdotp). 
\end{equation} 
\noindent By orthogonality, we have 
\begin{equation}
\label{relation1}
\begin{split}
\| NS(v_{0,n})(\tau) \|^{2}_{\dot{H}^s} &\geqslant  \| NS(\varphi^{j_{0}})(\tau) \|^{2}_{\dot{H}^s}    \,\, +\, \sum_{\stackrel{j \in \mathcal{J}_{1}, j \neq j_{0}}{ \tau^{j}_{*} > 1}} \| NS(\varphi^{j})(\tau) \|^{2}_{\dot{H}^s}\, +  \, +\|e^{\tau \Delta}\psi_{n}^J \|^{2}_{\dot{H}^s}\,\, + |\gamma_{n}^{J}(\tau)|. 
\end{split}
\end{equation} 
\noindent We want to prove that~$\ds{\liminf_{\tau \to 1} \,\, (1 -\tau )^{\frac{1}{\sigma_{s}}}\,  \| NS(\varphi^{j_{0}})(\tau) \|_{\dot{H}^s} \geqslant m_{c}}$. By definition of~$m_{c}$, this will imply that~$\ds{\liminf_{\tau \to 1} \,\, (1 -\tau )^{\frac{1}{\sigma_{s}}}\,  \| NS(\varphi^{j_{0}})(\tau) \|_{\dot{H}^s} = m_{c}}$. Let us assume that is not the case. Therefore,
$$ \exists \alpha_{0} >0, \forall \varepsilon >0, \,\, \exists \tau_{\varepsilon},\,\,\, \hbox{such that} \,\,\, 0<  (1 -\tau_{\varepsilon} )^{\frac{2}{\sigma_{s}}}\, < \varepsilon \,\,  \hbox{and} \,\, (1 -\tau_{\varepsilon} )^{\frac{2}{\sigma_{s}}}\,  \| NS(u_{0,n})(\tau_{\varepsilon}) \|^{2}_{\dot{H}^s} \leqslant m_{c}^2 - \alpha_{0}.$$
\noindent From (\ref{relation1}), we deduce that
\begin{equation*}
\begin{split}
(1 -\tau_{\varepsilon} )^{\frac{2}{\sigma_{s}}}\, \| NS(v_{0,n})(\tau_{\varepsilon}) \|^{2}_{\dot{H}^s} &=  (1 -\tau_{\varepsilon} )^{\frac{2}{\sigma_{s}}}\, \| NS(\varphi^{j_{0}})(\tau_{\varepsilon}) \|^{2}_{\dot{H}^s}    \,\,  +\,  (1 -\tau_{\varepsilon} )^{\frac{2}{\sigma_{s}}}\, \Bigl\{   \sum_{\stackrel{j \in \mathcal{J}_{1}, j \neq j_{0}}{ \tau^{j}_{*} > 1}}  \| NS(\varphi^{j})(\tau_{\varepsilon}) \|^{2}_{\dot{H}^s}  \, \\ &\quad+\, \|e^{\tau_{\varepsilon} \Delta}\psi_{n}^J \|^{2}_{\dot{H}^s}\, + \, \vert\gamma_{n}^{J}(\tau_{\varepsilon}) \vert\Bigr\}. 
\end{split}
\end{equation*} 
\noindent By hypothesis, $\ds{(1 -\tau_{\varepsilon} )^{\frac{1}{\sigma_{s}}}\, \| NS(v_{0,n})(\tau_{\varepsilon}) \|_{\dot{H}^s} \geqslant m_{c} - \varepsilon_{n} }$, and $\ds{1 -\tau_{\varepsilon}  \leqslant 1}$. Hence, we get
\begin{equation}
\begin{split}
\bigl(m_{c} - \varepsilon_{n}\bigr)^2 &\leqslant \,\, m_{c}^2 - \alpha_{0}    \,\, +\,  (1 -\tau_{\varepsilon} )^{\frac{2}{\sigma_{s}}}\, \Bigl\{   \sum_{\stackrel{j \in \mathcal{J}_{1}, j \neq j_{0}}{ \tau^{j}_{*} > 1}} \sup_{\tau \in [0,1]}\, \| NS(\varphi^{j})(\tau) \|^{2}_{\dot{H}^s}  \, +\, \|\psi_{n}^J \|^{2}_{\dot{H}^s}\Bigr\}\,\, + \vert\gamma_{n}^{J}(\tau_{\varepsilon})\vert . 
\end{split}
\end{equation} 
\noindent On the one hand, as profiles $\varphi^{j}$ have a lifespan $\tau^{j}_{*} > 1$, the quantity $\ds{\sup_{\tau \in [0,1]}\, \| NS(\varphi^{j})(\tau) \|^{2}_{\dot{H}^s} }$ is finite. On the other hand, by vertue of profile decomposition of the sequence $(v_{0,n})$, we have obviously that~$\ds{ \|\psi_{n}^J \|^{2}_{\dot{H}^s} \leqslant  \|v_{0,n} \|^{2}_{\dot{H}^s}}$. As we have proved that $(v_{0,n})$ is an element of the set~$\mathcal{E}_{c}$, we get in particular that~$\ds{\sup_{\tau < 1}\, (1-\tau)^{\frac{1}{\sigma_{s}}}\, \| NS(v_{0,n})(\tau) \|_{\dot{H}^s} = M_{c}}$, which leads to (at $\tau =0$) $\ds{\|v_{0,n} \|}_{\dot{H}^s} \leqslant  M_{c} $. Finally, for all ~$\tau_{\varepsilon}$,
$$  (1 -\tau_{\varepsilon} )^{\frac{2}{\sigma_{s}}}\, \Bigl\{   \sum_{\stackrel{j \in \mathcal{J}_{1}, j \neq j_{0}}{ \tau^{j}_{*} > 1}} \sup_{\tau \in [0,1]}\, \| NS(\varphi^{j})(\tau) \|^{2}_{\dot{H}^s}  \, +\, \|\psi_{n}^J \|^{2}_{\dot{H}^s}\Bigr\} \leqslant \frac{\alpha_{0}}{4},$$ we get
\begin{equation}
\begin{split}
\bigl(m_{c} - \varepsilon_{n}\bigr)^2 &\leqslant \,\, m_{c}^2 - \alpha_{0}    \,\, +\,  \frac{\alpha_{0}}{4}  + \vert\gamma_{n}^{J}(\tau_{\varepsilon})\vert . 
\end{split}
\end{equation} 
\noindent Now, by assumption of $\gamma_{n}^{J}$, we take the limit for $n$ and $J$ large enough, and we get
\begin{equation}
m_{c}^2 \leqslant m_{c}^2 \,-\, \frac{3\,\alpha_{0}}{4} \, +\, \frac{\alpha_{0}}{4},
\end{equation}
\noindent which is obviously absurd. Thus, we have proved that
$$ \liminf_{\tau \to 1} \,\,\, (1 -\tau )^{\frac{1}{\sigma_{s}}}\,  \| NS(\varphi^{j_{0}})(\tau) \|_{\dot{H}^s} = m_{c}.$$
This concludes the proof of Lemma \ref{sup inf critical lemma}. 
\end{proof}

\section{Structure Lemma for Navier-Stokes solutions with bounded data}
\noindent The sequence $(v_{0,n})_{n \geqslant 0}$ be a bounded sequence of initial data in $\dot{H}^s$. Thanks to Theorem \ref{theo profiles}, $(v_{0,n})_{n \geqslant 0}$ can be written as follows, up to an extraction
$$v_{0,n}(x) = \sum_{j=0}^{J} \Lambda^{\frac{3}{p}}_{\lambda_{n,j},x_{n,j}}\varphi^{j}(x) + \psi_{n}^{J}(x),$$
\noindent which can be written as follows
\begin{equation}
 \begin{split}
 v_{0,n}(x) &= \sum_{\stackrel{j \in \mathcal{J}_{1}}{j \leqslant J}} \varphi^{j}(x-x_{n,j})
           + \sum_{\stackrel{j \in \mathcal{J}^{{}{c}}_{1}}{j \leqslant J}} \Lambda^{\frac{3}{p}}_{\lambda_{n,j},x_{n,j}}\varphi^{j}(x) + \psi_{n}^{J}(x).\\
 \end{split}
\end{equation}
\noindent Let $\eta >0$ be the parameter of rough cutting off frequencies. We define by $w_{\eta}(x)$ and $w_{{}^{c}\eta}(x)$ the elements which Fourier transform is given by
\begin{equation}
\label{notations cut off}
\widehat{w_{\eta}}(\xi) = \widehat{w}(\xi) 1_{\{\frac{1}{\eta} \leqslant |\xi| \leqslant \eta \}} \quad \hbox{and} \quad
\widehat{w_{{}^{c}\eta}}(\xi) = \widehat{w}(\xi) \bigl(1 -1_{\{\frac{1}{\eta} \leqslant |\xi| \leqslant \eta \}}\bigr).
\end{equation}
\vsd 
\noindent After rough cutting off frequencies with respect to the notations $(\ref{notations cut off})$ and  sorting profiles supported in the annulus $1_{\{\frac{1}{\eta}\leqslant |\xi| \leqslant \eta\}}$ according to their scale (thanks to the orthogonality property of scales and cores, given by Theorem \ref{theo profiles}). We get the following profile decomposition \\
\begin{equation}
\label{decomposition apres regularisation}
 \begin{split}
 v_{0,n}(x) &= \sum_{j \in \mathcal{J}_{1}} \varphi^{j}(x-x_{n,j})
           + \sum_{j \in \mathcal{J}_{0}} \Lambda^{\frac{3}{p}}_{\lambda_{n,j},x_{n,j}}\varphi^{j}_{\eta}(x) + \sum_{j \in \mathcal{J}_{\infty}} \Lambda^{\frac{3}{p}}_{\lambda_{n,j},x_{n,j}}\varphi^{j}_{\eta}(x) + \psi_{n,\eta}^{J}(x)\\
&\hbox{where} \quad \psi_{n,\eta}^{J}(x) \eqdefa \sum_{\stackrel{j \in \mathcal{J}^{{}{c}}_{1} \equiv \mathcal{J}_{0} \cup \mathcal{J}_{\infty} }{j \leqslant J}} \Lambda^{\frac{3}{p}}_{\lambda_{n,j},x_{n,j}}V^{j}_{{}^{c}\eta}(x)\,  +\, \psi_{n}^{J}(x), 
 \end{split}
\end{equation}
\\
\noindent for any $j$ in $\mathcal{J}_{1} \subset J$, $\lambda_{n,j} =1$,\,\, for any $j$ in $\mathcal{J}_{0}$, $\displaystyle{\lim_{n \to +\infty} \lambda_{n,j} = 0}$\,\, and for any ~$j$ in $\mathcal{J}_{\infty}$, ~$\displaystyle{\lim_{n \to +\infty} \lambda_{n,j} = +\infty}$.\\
\noindent As mentionned in the introduction, the whole Lemma \ref{lemme allure de la solution} has been already proved in \cite{P}, except for the orthogonality property of the Navier-stokes solution associated with such a sequence of initial data. Therefore, we refer the reader to \cite{P} for details of the proof and here, we focus on the "Pythagore property". Let us recall the notations 
$$U^{0}_{n,\eta}\eqdefa \sum_{j \in \mathcal{J}_{0}} \Lambda^{\frac{3}{p}}_{\lambda_{n,j},x_{n,j}}\varphi^{j}_{\eta} \quad \hbox{and} \quad U^{\infty}_{n,\eta}\eqdefa \sum_{j \in \mathcal{J}_{\infty}} \Lambda^{\frac{3}{p}}_{\lambda_{n,j},x_{n,j}}\varphi^{j}_{\eta}.$$
We recall some properties on profiles with small and large scale and remaining term. We refer the reader to \cite{P} to the proof of the two propositions below.  
\begin{prop}\sl{ 
\label{smallbigscaling}
$$\hbox{For any} \,\, s_{1}<s,\,\,\hbox{for any} \,\, \eta >0, \,\, \hbox{for any} \,\, j \in \mathcal{J}_{0}, \,\, (\hbox{e.g} \,\, \lim_{n \to +\infty} \lambda_{n,j} = 0), \,\, \hbox{then} \,\, \lim_{n \to +\infty}\bigl\|U^{0}_{n,\eta}\bigr\|_{\dot{H}^{s_{1}}} = 0.$$
$$\hbox{For any} \,\, s_{2}>s,\,\,\,\,\hbox{for any} \,\, \eta >0, \,\, \hbox{for any} \,\, j \in \mathcal{J}_{\infty}, \,\, (\hbox{e.g} \,\, \lim_{n \to +\infty} \lambda_{n,j} = +\infty), \,\, \hbox{then} \,\, \lim_{n \to +\infty}\bigl\|U^{\infty}_{n,\eta}\bigr\|_{\dot{H}^{s_{2}}} = 0.$$
}\end{prop}
\noindent Concerning the remaining term, we can show it tends to $0$, thanks to Lebesgue Theorem.
\begin{prop}\sl{
\label{reste perturb\'e petit}
 $$ \lim_{J \to +\infty} \lim_{\eta \to +\infty} \limsup_{n \to +\infty} \| \psi_{n,\eta}^{J}\|_{L^p} = 0.$$
}\end{prop}

\noindent \textit{Continuation of Proof of Lemma \ref{lemme allure de la solution}.} By vertue of (\ref{decomposition de la solution}) in Lemma \ref{lemme allure de la solution}, it seems clear that for any~$t < \tilde{T}$
\begin{equation*}
 \begin{split}
 \label{Pyt 2}
  \| NS(v_{0,n})(t,\cdot) \|^2_{\dot{H}^s} &= \Bigl\| \sum_{j \in \mathcal{J}_{1}} NS(\varphi^{j})(t,\cdot-x_{n,j}) \Bigr\|^{2}_{\dot{H}^s} + \Bigl\| e^{t\Delta}\Bigl( \sum_{\stackrel{j \in \mathcal{J}^{{}{c}}_{1}}{j \leqslant J}} \Lambda^{\frac{3}{p}}_{\lambda_{n,j},x_{n,j}}\varphi^{j}(x) + \psi_{n}^{J}\Bigr)\Bigr\|^{2}_{\dot{H}^s}\\ &+ \| R_{n}^{J}(t,\cdot) \|^{2}_{\dot{H}^s}\,
  + 2\, \Bigl( \sum_{j \in \mathcal{J}_{1}}  NS(\varphi^{j})(t,\cdot-x_{n,j}) \mid e^{t\Delta}\Bigl( \sum_{\stackrel{j \in \mathcal{J}^{{}{c}}_{1}}{j \leqslant J}} \Lambda^{\frac{3}{p}}_{\lambda_{n,j},x_{n,j}}\varphi^{j}(x) + \psi_{n}^{J}\Bigr)\Bigr)_{\dot{H}^s}\\& + 2\, \Bigl( \sum_{j \in \mathcal{J}_{1}}  NS(\varphi^{j})(t,\cdot-x_{n,j}) \mid R_{n}^{J} \Bigr)_{\dot{H}^s} 
  + 2\, \Bigl( e^{t\Delta}\Bigl( \sum_{\stackrel{j \in \mathcal{J}^{{}{c}}_{1}}{j \leqslant J}} \Lambda^{\frac{3}{p}}_{\lambda_{n,j},x_{n,j}}\varphi^{j}(x) + \psi_{n}^{J}\Bigr) \,  \mid \, R_{n}^{J} \Bigr)_{\dot{H}^s}.
 \end{split}
\end{equation*}
\noindent Therefore, proving (\ref{Pythagore}) is equivalent to prove Propositions \ref{prop orthogonali\'e profiles echelle 1} and \ref{prop orthogonali\'e profiles echelle 0 et infini} below. Both of them essentially stem from the orthogonality of cores and a compactness argument. 
\begin{prop}\sl{
\noindent Let $\varepsilon >0$. Then, for any $t \in [0,\tilde{T}-\varepsilon] $, 
\label{prop orthogonali\'e profiles echelle 1}
\begin{equation}
 \Bigl\| \, \sum_{j \in \mathcal{J}_{1}}  NS(\varphi^{j})(t,\cdotp - x_{n,j}) \Bigr\|^{2}_{\dot{H}^s} =   \sum_{j \in \mathcal{J}_{1}} \bigl \| NS(\varphi^{j})(t,\cdotp) \bigr\|^{2}_{\dot{H}^s} + \gamma_{n,\varepsilon}(t),
\end{equation}
with $\ds{\lim_{n \to +\infty} \sup_{t \in [0,\tilde{T}-\varepsilon]} \vert \gamma_{n,\varepsilon}(t)  \vert =0.}$
}\end{prop}
\begin{proof}
Once again, we developp the square of $\ds{\dot{H}^s}$-norm and we get for any $\ds{t < \tilde{T}}$ 
\begin{equation*}
 \begin{split}
 &\Bigl\| \sum_{j \in \mathcal{J}_{1}} NS(\varphi^{j})(t, \cdotp-x_{n,j}) \Bigr\|^{2}_{\dot{H}^s}
  = \sum_{j \in \mathcal{J}_{1}} \bigl\|  NS(\varphi^{j})(t, \cdotp-x_{n,j}) \bigr\|^{2}_{\dot{H}^s}\\ &\qquad \qquad \qquad \qquad \qquad \qquad \qquad  + 2\, \sum_{\stackrel{(j,k) \in \mathcal{J}_{1} \times \mathcal{J}_{1}}{j \neq k}} \left( \Lambda^s\, NS(\varphi^{j})(t, \cdotp-x_{n,j}) \mid \Lambda^s\, NS(\varphi^{k})(t, \cdotp-x_{n,k}) \right)_{L^2},\\
   \end{split}
\end{equation*}
\noindent where $\Lambda = \sqrt{-\Delta}$. 
\noindent Let $\varepsilon>0$. Then, for any t in $[0, \widetilde{T} - \varepsilon]$, we get
\begin{equation*}
 \begin{split}
  \Bigl\| \sum_{j \in \mathcal{J}_{1}} NS(\varphi^{j})(t, \cdotp-x_{n,j}) \Bigr\|^{2}_{\dot{H}^s}
  &= \sum_{j \in \mathcal{J}_{1}} \bigl\|  NS(\varphi^{j})(t, \cdotp)   \bigr\|^{2}_{\dot{H}^s}  + 2\, \sum_{\stackrel{(j,k) \in \mathcal{J}_{1} \times \mathcal{J}_{1}}{j \neq k}} \Gamma^{s,j,k}_{\varepsilon,n},\\
 \end{split}
\end{equation*}
\noindent where $\ds{\Gamma^{s,j,k}_{\varepsilon,n}
 \eqdefa \left( \Lambda^s\, NS(\varphi^{j})(t, \cdotp-x_{n,j}) \mid \Lambda^s\, NS(\varphi^{k})(t, \cdotp-x_{n,k}) \right)_{L^2}}$.\\

\noindent We denote by 
$$K_{\varepsilon}^{J} \eqdefa \bigcup_{j \in J}\, \Lambda^s \,NS(\varphi^{j})([0, \widetilde{T} - \varepsilon]).$$
\noindent By vertue of the continuity of the map~$\ds{ t \in [0, \widetilde{T} - \varepsilon] \mapsto  \Lambda^s \,NS(\varphi^{j})(t, \cdotp)\, \in L^2 }$, we deduce that $K_{\varepsilon}^{J}$ is compact (and thus precompact) in $L^2$. It means that it can be covered by a finite open ball with an arbitrarily radius $\alpha >0$. Let $\alpha$ be a positive radius. There exists an integer~$N_{\alpha}$, and there exists~$(\theta_{\ell})_{1 \leqslant \ell \leqslant N_{\alpha}}$ some elements of $\mathcal{D}(\R^3)$, such that
\begin{equation}
\ds{K_{\varepsilon}^{J} \subset \bigcup_{\ell =1}^{ N_{\alpha}} \, B(\theta_{\ell},\alpha)}.
\end{equation}
\noindent Let us come back to the proof of \ref{prop orthogonali\'e profiles echelle 1}. Thanks to the previous remark, we approach each profil~$\Lambda^s \,NS(\varphi^{j})(t, \cdotp)$ (resp. $\Lambda^s \,NS(\varphi^{k})(t, \cdotp)$) by a smooth function: e.g there exists a integer~$\ell \in \{1,\cdotp\cdotp\cdotp N_{\alpha} \} $ and there exists a function $\theta_{\ell(j,t)}$ (resp.~$\theta_{\ell(k,t)}$) in $\mathcal{D}(\R^3)$  and we get 
\begin{equation}
\begin{split}
\Gamma^{s,j,k}_{\varepsilon,n} &= \left( \Lambda^s\, NS(\varphi^{j})(t, \cdotp-x_{n,j}) - \theta_{\ell(j,t)}(\cdotp-x_{n,j}) \mid \Lambda^s\, NS(\varphi^{k})(t, \cdotp-x_{n,k}) - \theta_{\ell(k,t)}(\cdotp-x_{n,k}) \right)_{L^2}\\
&+ \left( \Lambda^s\, NS(\varphi^{j})(t, \cdotp-x_{n,j}) - \theta_{\ell(j,t)}(\cdotp-x_{n,j}) \mid  \theta_{\ell(k,t)}(\cdotp-x_{n,k})  \right)_{L^2}\\
&+ \left( \theta_{\ell(j,t)}(\cdotp-x_{n,j})  \mid \Lambda^s\, NS(\varphi^{k})(t, \cdotp-x_{n,k}) - \theta_{\ell(k,t)}(\cdotp-x_{n,k}) \right)_{L^2}\\
&+ \left(  \theta_{\ell(j,t)}(\cdotp-x_{n,j})  \mid  \theta_{\ell(k,t)}(\cdotp-x_{n,k}) \right)_{L^2}.\\
\end{split}
\end{equation}
\noindent The three first terms in the right-hand side of the above estimate tend uniformly (in time) to $0$, by vertue of Cauchy-Schwarz and the translation-invariance of the $\dot{H}^s$-norm (we just perform the estimate for the first term, the others are similar). For any $t \in [0,\tilde{T}-\varepsilon]$
\begin{equation}
\begin{split}
\Bigl( \Lambda^s\, NS(\varphi^{j})(t, \cdotp-x_{n,j}) - \theta_{\ell(j,t)}(\cdotp-x_{n,j}) &\mid \Lambda^s\, \bigl( NS(\varphi^{k})(t, \cdotp-x_{n,k}) - \theta_{\ell(k,t)}(\cdotp-x_{n,k}) \bigl)\Bigr)_{L^2}\\ &\leqslant \| \Lambda^s\,  NS(\varphi^{j})(t) - \theta_{\ell(j,t)} \|_{L^2} \,\, \| \Lambda^s\,  NS(\varphi^{k})(t) - \theta_{\ell(k,t)} \|_{L^2}\\
&\leqslant \alpha^2.
\end{split}
\end{equation}
\noindent Therefore, for any $\alpha >0$, we have
\begin{equation}
\sup_{t \in [0,\tilde{T}-\varepsilon]} \, \left( \Lambda^s\, NS(\varphi^{j})(t, \cdotp-x_{n,j}) - \theta_{\ell(j,t)}(\cdotp-x_{n,j}) \mid \Lambda^s\, NS(\varphi^{k})(t, \cdotp-x_{n,k}) - \theta_{\ell(k,t)}(\cdotp-x_{n,k}) \right)_{L^2} \leqslant \alpha^2. 
\end{equation}
\noindent For the last term $\ds{\left(  \theta_{\ell(j,t)}(\cdotp-x_{n,j})  \mid \,\theta_{\ell(k,t)}(\cdotp-x_{n,k}) \right)_{L^2}}$, we have 
$$ \left(  \theta_{\ell(j,t)}(\cdotp-x_{n,j})  \mid \,\theta_{\ell(k,t)}(\cdotp-x_{n,k}) \right)_{L^2} = \int_{\R^3}\, \theta_{\ell(j,t)}(x)\, \theta_{\ell(k,t)}(x + x_{n,j} - x_{n,k})\, dx.$$
\noindent It follows immediately that the above term tends to $0$, when $n$ tend to $+\infty$, by vertue of Lebesgue theorem combining with the orthogonality property of cores(e.g. $\ds{\lim_{n \to \infty} |x_{n,j}-x_{n,k}| = +\infty}$). 
\noindent To sum up, we have proved that $\Gamma^{s,j,k}_{\varepsilon,n}$ tends to $0$ when $n$ tends to $+\infty$, uniformly in time. This concludes the proof of Proposition \ref{prop orthogonali\'e profiles echelle 1}.

\end{proof}
\noindent Concerning the crossed-terms in the profile decomposition, we have to prove they are also negligable, uniformly in time. That is the point in the following proposition.  
\begin{prop}\sl{
\noindent Let $\varepsilon>0$, We denote by 
$$ I_{n}(t,\cdotp) \eqdefa \Bigl( \sum_{j \in \mathcal{J}_{1}}  NS(\varphi^{j})(t,\cdot-x_{n,j}) \mid e^{t\Delta}\Bigl( \sum_{\stackrel{j \in \mathcal{J}^{{}{c}}_{1}}{j \leqslant J}} \Lambda^{\frac{3}{p}}_{\lambda_{n,j},x_{n,j}}\varphi^{j}(x) + \psi_{n}^{J}\Bigr)\Bigr)_{\dot{H}^s},$$
\label{prop orthogonali\'e profiles echelle 0 et infini}
\begin{equation}
\label{estimate5}
\begin{split}
\hbox{then, one has} \quad \lim_{J \to +\infty}\lim_{\eta \to +\infty} \lim_{n \to +\infty} \, \sup_{t \in [0, \tilde{T}- \varepsilon] } \, I_{n}(t,\cdotp) = 0,\\
\end{split}
\end{equation}
\begin{equation}
\label{estimate6}
\begin{split}
\lim_{J \to +\infty} \lim_{n \to +\infty} \, \sup_{t \in [0, \tilde{T}- \varepsilon] }\Bigl( \sum_{j \in \mathcal{J}_{1}}  NS(\varphi^{j})(t,\cdot-x_{n,j}) \mid R_{n}^{J}(t) \Bigr)_{\dot{H}^s}  = 0,
\end{split}
\end{equation}
\begin{equation}
\label{estimate7}
\begin{split}
\lim_{J \to +\infty}\lim_{\eta \to +\infty} \lim_{n \to +\infty} \, \sup_{t \in [0, \tilde{T}- \varepsilon] } \, \Bigl( e^{t\Delta}\Bigl( \sum_{\stackrel{j \in \mathcal{J}^{{}{c}}_{1}}{j \leqslant J}} \Lambda^{\frac{3}{p}}_{\lambda_{n,j},x_{n,j}}\varphi^{j}(x) + \psi_{n}^{J}\Bigr)\mid R_{n}^{J}(t) \Bigr)_{\dot{H}^s} = 0.
\end{split} 
\end{equation}
}\end{prop}

\begin{proof}
\noindent Let us start by proving (\ref{estimate5}). We shall use once again an approximation argument. Let us define $$\Lambda_{\varepsilon}^{J} \eqdefa \bigcup_{j \in J}\,  \,NS(\varphi^{j})([0, \widetilde{T} - \varepsilon]).$$
\noindent By vertue of the continuity of the map~$\ds{ t \in [0, \widetilde{T} - \varepsilon] \mapsto  NS(\varphi^{j})(t, \cdotp)\, \in \dot{H}^s }$, we deduce that $\Lambda_{\varepsilon}^{J}$ is compact (and thus precompact) in $\dot{H}^s$. It means that it can be covered by a finite open ball with an arbitrarily radius $\beta >0$. Let $\beta$ be a positive radius. There exists an integer~$N_{\beta}$, and there exists~$(\chi_{\ell})_{1 \leqslant \ell \leqslant N_{\beta}}$ some elements of $\mathcal{D}(\R^3)$, such that
\begin{equation}
\ds{\Lambda_{\varepsilon}^{J} \subset \bigcup_{\ell =1}^{ N_{\beta}} \, B(\chi_{\ell},\beta)}.
\end{equation}
\noindent Let us come back to the proof of (\ref{estimate5}). Same arguments as previously imply there exists an integer~$\ell \in \{  1 \cdotp \cdotp \cdotp N_{\beta} \}$ and a smooth function~$\chi_{\ell(t,j)}$ in $\mathcal{D}(\R^3)$ such that 
\begin{equation}
\begin{split}
I_{n}(t,\cdotp) &\eqdefa \Bigl( \sum_{j \in \mathcal{J}_{1}}  NS(\varphi^{j})(t,\cdot-x_{n,j}) \mid e^{t\Delta}\Bigl( \sum_{\stackrel{j \in \mathcal{J}^{{}{c}}_{1}}{j \leqslant J}} \Lambda^{\frac{3}{p}}_{\lambda_{n,j},x_{n,j}}\varphi^{j} + \psi_{n}^{J}\Bigr)\Bigr)_{\dot{H}^s}\\
&= \Bigl( \sum_{j \in \mathcal{J}_{1}}  NS(\varphi^{j})(t,\cdot-x_{n,j}) - \chi_{\ell(t,j)}(\cdotp - x_{n,j}) \mid e^{t\Delta}\Bigl( \sum_{\stackrel{j \in \mathcal{J}^{{}{c}}_{1}}{j \leqslant J}} \Lambda^{\frac{3}{p}}_{\lambda_{n,j},x_{n,j}}\varphi^{j}+ \psi_{n}^{J}\Bigr)\Bigr)_{\dot{H}^s}\\
&\qquad+ \Bigl( \sum_{j \in \mathcal{J}_{1}}  \chi_{\ell(t,j)}(\cdotp - x_{n,j}) \mid e^{t\Delta}\Bigl( \sum_{\stackrel{j \in \mathcal{J}^{{}{c}}_{1}}{j \leqslant J}} \Lambda^{\frac{3}{p}}_{\lambda_{n,j},x_{n,j}}\varphi^{j} + \psi_{n}^{J}\Bigr)\Bigr)_{\dot{H}^s}.
\end{split}
\end{equation}
\noindent As $\ds{ \bigl\| e^{t\Delta}\Bigl( \sum_{\stackrel{j \in \mathcal{J}^{{}{c}}_{1}}{j \leqslant J}} \Lambda^{\frac{3}{p}}_{\lambda_{n,j},x_{n,j}}\varphi^{j} + \psi_{n}^{J}\Bigr) \bigr\|_{\dot{H}^s}  \leqslant \|  v_{0,n} \|_{\dot{H}^s}    }$, we infer that 
\begin{equation}
\begin{split}
I_{n}(t,\cdotp) &\leqslant  |\mathcal{J}_{1}|\, \beta \, \bigl\| v_{0,n}  \bigr\|_{\dot{H}^s} +\, \Bigl( \sum_{j \in \mathcal{J}_{1}}  \chi_{\ell(t,j)}(\cdotp - x_{n,j}) \mid e^{t\Delta}\Bigl( \sum_{\stackrel{j \in \mathcal{J}^{{}{c}}_{1}}{j \leqslant J}} \Lambda^{\frac{3}{p}}_{\lambda_{n,j},x_{n,j}}\varphi^{j} + \psi_{n}^{J}\Bigr)\Bigr)_{\dot{H}^s}\\
\end{split}
\end{equation}

\noindent Concerning the second part of above inequality, we shall use the splitting with respect to the parameter of cut off $\eta$. We refer the reader to the beginning of this section for notations.
$$ \Bigl( \sum_{j \in \mathcal{J}_{1}}  \chi_{\ell(t,j)}(\cdotp - x_{n,j}) \mid e^{t\Delta}\Bigl( \sum_{\stackrel{j \in \mathcal{J}^{{}{c}}_{1}}{j \leqslant J}} \Lambda^{\frac{3}{p}}_{\lambda_{n,j},x_{n,j}}\varphi^{j} + \psi_{n}^{J}\Bigr)\Bigr)_{\dot{H}^s} = I^{1}_{n}(t,\cdotp) + I^{2}_{n}(t,\cdotp) + I^{3}_{n}(t,\cdotp),$$
\begin{equation*}
\begin{split}
\hbox{where} \quad &I^{1}_{n}(t,\cdotp) =  \sum_{j \in \mathcal{J}_{1}}  \left(\chi_{\ell(t,j)}(\cdotp - x_{n,j}) \mid e^{t\Delta} U^{0}_{n,\eta}\right)_{\dot{H}^s}\quad \hbox{;} \quad I^{2}_{n}(t,\cdotp) = \sum_{j \in \mathcal{J}_{1}}  \left(\chi_{\ell(t,j)}(\cdotp - x_{n,j}) \mid e^{t\Delta} U^{\infty}_{n,\eta}\right)_{\dot{H}^s}\\&\qquad \qquad \qquad \qquad \hbox{and} \quad I^{3}_{n}(t,\cdotp) = \sum_{j \in \mathcal{J}_{1}}  \left(\chi_{\ell(t,j)}(\cdotp - x_{n,j}) \mid e^{t\Delta} \psi_{n, \eta}^{J}\right)_{\dot{H}^s}.
\end{split}
\end{equation*} 
\noindent Let us start with $I^{1}_{n}(t,\cdotp)$. One has
\begin{equation*}
\begin{split}
|I^{1}_{n}(t,\cdotp)| &\leqslant  |\mathcal{J}_{1}|\, \|  \chi_{\ell(t,j)}  \|_{\dot{H}^{2s-s_{1}}} \, \|  e^{t\Delta} U^{0}_{n,\eta} \|_{\dot{H}^{s_{1}}}\\
&\leqslant |\mathcal{J}_{1}|\,\|  \chi_{\ell(t,j)}  \|_{\dot{H}^{2s-s_{1}}} \, \|  U^{0}_{n,\eta}\|_{\dot{H}^{s_{1}}}.
\end{split}
\end{equation*}

\noindent Proposition \ref{smallbigscaling} (for $\eta$ and $j \in \mathcal{J}_{1}$ fixed) implies thus $\ds{\lim_{n \to +\infty} \, \sup_{t \in [0, \tilde{T}- \varepsilon] } |I^{1}_{n}(t,\cdotp)| = 0}$.\\
\noindent Concerning profiles with large scale, the proof is similar and we get for any $t \in [0, \tilde{T}- \varepsilon]$
\begin{equation}
\begin{split}
|I^{2}_{n}(t,\cdotp)| &\leqslant\,   |\mathcal{J}_{1}|\, \|  \chi_{\ell(t,j)} \|_{\dot{H}^{2s-s_{2}}} \,  \bigl\|  U^{\infty}_{n,\eta}(x)  \,\bigr\|_{\dot{H}^{s_{2}}}.
\end{split}
\end{equation}
\noindent Once again, Proposition \ref{smallbigscaling} implies the result : $\ds{\lim_{n \to +\infty} \, \sup_{t \in [0, \tilde{T}- \varepsilon] } |I^{2}_{n}(t,\cdotp)| = 0}$.\\ \noindent Concerning the last term $I^{3}_{n}$, H\"older inequality with $\ds{\frac{1}{p} + \frac{1}{p'} =1}$ yields
\begin{equation*}
\begin{split}
|I^{3}_{n}(t,\cdotp)|
&\leqslant  \bigl|\left(\Lambda^{2s}\,\chi_{\ell(t,j)} \mid e^{t\Delta} \psi_{n, \eta}^{J}(\cdotp + x_{n,j})\right)_{L^2}\bigr| \\
&\leqslant \| \Lambda^{2s}\, \chi_{\ell(t,j)} \|_{L^{p'}} \, \|  e^{t\Delta} \psi_{n, \eta}^{J}(\cdotp + x_{n,j}) \|_{L^p}  \bigr).\\
\end{split}
\end{equation*}
\noindent By translation invariance of the $L^p$-norm and estimate on the heat equation, we get 
\begin{equation}
\begin{split}
|I^{3}_{n}(t,\cdotp)| 
&\leqslant  \| \Lambda^{2s}\, \chi_{\ell(t,j)} \|_{L^{p'}} \, \|  \psi_{n, \eta}^{J}\|_{L^p}.
\end{split}
\end{equation}
\noindent Obviously the term $\ds{\| \psi_{n, \eta}^{J}\|_{\dot{H}^s}}$ is bounded by profiles hypothesis and the term $\ds{\| \Lambda^{2s}\, \chi \chi_{\ell(t,j)} \|_{L^{p'}}}$ is bounded too, since the function $\chi$ is as regular as we need. By vertue of Proposition \ref{reste perturb\'e petit}, the term $\ds{\|  \psi_{n, \eta}^{J}\|_{L^p}}$ is small in the sense of for any $\varepsilon >0$, there exists an integer $N_{0} \in N$, such that for any $n \geqslant N_{0}$, there exists $\tilde{\eta} >0$ and $\tilde{J} \geqslant 0$, such that for any $\eta \geqslant \tilde{\eta}$ and for any $J \geqslant \tilde{J}$, we have $\ds{\|  \psi_{n, \eta}^{J}\|_{L^p} \leqslant \varepsilon}$. As a result, we get for any  $$  \lim_{J \to +\infty}\lim_{\eta \to +\infty} \lim_{n \to +\infty} \, \sup_{t \in [0, \tilde{T}- \varepsilon] } |I^{3}_{n}(t,\cdotp)| = 0.$$ This ends up the proof of estimate (\ref{estimate5}).

\medbreak 
\noindent Concerning the proof of (\ref{estimate6}) and (\ref{estimate7}), the proof is very close in both cases and relies on the fact that the error term $R_{n}^{J}$ tends to $0$ in the $L^{\infty}_{T}(\dot{H}^s)$-norm. For any $t \in [0, \tilde{T}- \varepsilon]$, we have
\begin{equation}
\begin{split}
\bigl| \bigl( \sum_{j \in \mathcal{J}_{1}}  NS(\varphi^{j})(t,\cdot-x_{n,j}) \mid R_{n}^{J} \bigr)_{\dot{H}^s} \bigr| &\leqslant \sum_{j \in \mathcal{J}_{1}} \bigl| \bigl(NS(\varphi^{j})(t,\cdotp) \mid R_{n}^{J}(t,\cdotp + x_{n,j}) \bigr)_{\dot{H}^s}    \bigr |\\
&\leqslant |\mathcal{J}_{1}|\, \| NS(\varphi^{j})(t,\cdotp) \|_{L^{\infty}_{T}(\dot{H}^s)} \,  \| R_{n}^{J}(t,\cdotp) \|_{L^{\infty}_{T}(\dot{H}^s)}. 
\end{split}
\end{equation} 
\noindent Obviously, the term $\ds{\| NS(\varphi^{j})(t,\cdotp) \|_{L^{\infty}_{T}(\dot{H}^s)}}$ is bounded since $t \in [0, \tilde{T}- \varepsilon]$. As a result, Lemma \ref{lemme allure de la solution} implies that
$$  \lim_{J \to +\infty} \lim_{n \to +\infty} \, \sup_{t \in [0, \tilde{T}- \varepsilon] } \mid\bigl( \sum_{j \in \mathcal{J}_{1}}  NS(\varphi^{j})(t,\cdot-x_{n,j}) \mid R_{n}^{J} \bigr)_{\dot{H}^s}\mid = 0.$$

\noindent As far as estimate (\ref{estimate7}) is concerned, the idea is the same. For any $t \in [0, \tilde{T}- \varepsilon]$,
\begin{equation}
\begin{split}
\bigl| \bigl( e^{t\Delta}\Bigl( \sum_{\stackrel{j \in \mathcal{J}^{{}{c}}_{1}}{j \leqslant J}} \Lambda^{\frac{3}{p}}_{\lambda_{n,j},x_{n,j}}\varphi^{j}(x) + \psi_{n}^{J}\Bigr)\mid R_{n}^{J} \bigr)_{\dot{H}^s} \bigr|  &\leqslant \bigl|\bigl( e^{t\Delta}\Bigl( \sum_{\stackrel{j \in \mathcal{J}^{{}{c}}_{1}}{j \leqslant J}} \Lambda^{\frac{3}{p}}_{\lambda_{n,j},x_{n,j}}\varphi^{j}(x) + \psi_{n}^{J}\Bigr) \mid R_{n}^{J} \bigr)_{\dot{H}^s}    \bigr |\\
&\leqslant \| e^{t\Delta}\Bigl( \sum_{\stackrel{j \in \mathcal{J}^{{}{c}}_{1}}{j \leqslant J}} \Lambda^{\frac{3}{p}}_{\lambda_{n,j},x_{n,j}}\varphi^{j}(x) + \psi_{n}^{J}\Bigr) \|_{L^{\infty}_{T}(\dot{H}^s)}  \| R_{n}^{J} \|_{L^{\infty}_{\tilde{T}- \varepsilon}(\dot{H}^s)} \\
&\leqslant \|U^{0}_{n,\eta} + U^{\infty}_{n,\eta} + \psi_{n, \eta}^{J}\|_{\dot{H}^s} \, \| R_{n}^{J} \|_{L^{\infty}_{\tilde{T}- \varepsilon}(\dot{H}^s)}.
\end{split}
\end{equation} 
\noindent Thanks to profile decomposition (\ref{decomposition apres regularisation}), we get 
\begin{equation}
\begin{split}
\|U^{0}_{n,\eta} + U^{\infty}_{n,\eta} + \psi_{n, \eta}^{J}\|^{2}_{\dot{H}^s} &\leqslant \| v_{0,n}  \|^{2}_{\dot{H}^s} + \circ(1).
\end{split}
\end{equation}
\noindent Thus, finally we get
\begin{equation}
\begin{split}
\bigl| \bigl( e^{t\Delta}\Bigl( \sum_{\stackrel{j \in \mathcal{J}^{{}{c}}_{1}}{j \leqslant J}} \Lambda^{\frac{3}{p}}_{\lambda_{n,j},x_{n,j}}\varphi^{j}(x) + \psi_{n}^{J}\Bigr)\mid R_{n}^{J} \bigr)_{\dot{H}^s} | &\leqslant  C\, \bigl( \bigl\| v_{0,n} \bigr\|^{2}_{\dot{H}^s} + \circ(1)  \bigr) \, \| R_{n}^{J} \|_{L^{\infty}_{\tilde{T}- \varepsilon}(\dot{H}^s)}.
\end{split}
\end{equation}
\noindent We end up the proof as before, thanks to the hypothesis on $R_{n}^{J}$. This completes the proof of Proposition~\ref{prop orthogonali\'e profiles echelle 0 et infini} and thus Lemma~\ref{lemme allure de la solution}. 
\end{proof}

\medbreak

\end{document}